\documentclass[final,1p]{elsarticle}
\usepackage{amsmath,amssymb}
\usepackage[all]{xy}
\usepackage{latexsym}
\usepackage{amsthm,a4wide,color}
\usepackage{amsmath,amscd,verbatim}
\usepackage{hyperref}
\input diagxy

\theoremstyle{plain}
  \newtheorem{thm}{Theorem}[section]
  \newtheorem{lem}[thm]{Lemma}
  \newtheorem{prop}[thm]{Proposition}
  \newtheorem{cor}[thm]{Corollary}
\theoremstyle{definition}
  \newtheorem{defn}[thm]{Definition}
  \newtheorem{ques}[thm]{Question}
  \newtheorem{exmp}[thm]{Example}
  \newtheorem{rem}[thm]{Remark}

\makeatletter \def\ps@pprintTitle{  \let\@oddhead\@empty  \let\@evenhead\@empty  \def\@oddfoot{\centerline{\thepage}}  \let\@evenfoot\@oddfoot} \makeatother

\begin{document}

\newcommand{\oto}{{\lra\hspace*{-3.1ex}{\circ}\hspace*{1.9ex}}}
\newcommand{\lam}{\lambda}
\newcommand{\da}{\downarrow}
\newcommand{\Da}{\Downarrow\!}
\newcommand{\D}{\Delta}
\newcommand{\ua}{\uparrow}
\newcommand{\ra}{\rightarrow}
\newcommand{\la}{\leftarrow}
\newcommand{\Lra}{\Longrightarrow}
\newcommand{\Lla}{\Longleftarrow}
\newcommand{\rat}{\!\rightarrowtail\!}
\newcommand{\up}{\upsilon}
\newcommand{\Up}{\Upsilon}
\newcommand{\ep}{\epsilon}
\newcommand{\ga}{\gamma}
\newcommand{\Ga}{\Gamma}
\newcommand{\Lam}{\Lambda}
\newcommand{\CF}{{\cal F}}
\newcommand{\CG}{{\cal G}}
\newcommand{\CH}{{\cal H}}
\newcommand{\CI}{{\cal I}}
\newcommand{\CB}{{\cal B}}
\newcommand{\CT}{{\cal T}}
\newcommand{\CS}{{\cal S}}
\newcommand{\CV}{{\cal V}}
\newcommand{\CP}{{\cal P}}
\newcommand{\CQ}{{\cal Q}}
\newcommand{\mq}{\mathcal{Q}}
\newcommand{\cu}{{\underline{\cup}}}
\newcommand{\ca}{{\underline{\cap}}}
\newcommand{\nb}{{\rm int}}
\newcommand{\Si}{\Sigma}
\newcommand{\si}{\sigma}
\newcommand{\Om}{\Omega}
\newcommand{\bm}{\bibitem}
\newcommand{\bv}{\bigvee}
\newcommand{\bw}{\bigwedge}
\newcommand{\lra}{\longrightarrow}
\newcommand{\tl}{\triangleleft}
\newcommand{\tr}{\triangleright}
\newcommand{\dda}{\downdownarrows}
\newcommand{\dia}{\diamondsuit}
\newcommand{\y}{{\bf y}}
\newcommand{\colim}{{\rm colim}}
\newcommand{\fR}{R^{\!\forall}}
\newcommand{\eR}{R_{\!\exists}}
\newcommand{\dR}{R^{\!\da}}
\newcommand{\uR}{R_{\!\ua}}
\newcommand{\swa}{{\swarrow}}
\newcommand{\sea}{{\searrow}}
\newcommand{\bbA}{{\mathbb{A}}}
\newcommand{\bbB}{{\mathbb{B}}}
\newcommand{\bbC}{{\mathbb{C}}}
\numberwithin{equation}{section}
\renewcommand{\theequation}{\thesection.\arabic{equation}}

\begin{frontmatter}
\title{Yoneda completeness and  flat completeness of ordered fuzzy sets\tnoteref{F}}

\tnotetext[F]{This work is supported by National Natural Science Foundation of China (11371265, 11101297).}

\author{Wei Li}
\ead{mathli@foxmail.com}
\author{Hongliang Lai}
\ead{hllai@scu.edu.cn}
\author{Dexue Zhang\corref{cor}}
\ead{dxzhang@scu.edu.cn}
\cortext[cor]{Corresponding author, Tel: +86 13558678979}
\address{School of Mathematics, Sichuan University, Chengdu 610064, China}

\begin{abstract} This paper studies Yoneda completeness and flat completeness  of ordered fuzzy sets valued in the quantale obtained by endowing the unit interval  with  a  continuous triangular norm. Both of these notions are natural extension of   directed completeness in order theory to the fuzzy setting. Yoneda completeness requires every forward Cauchy net converges (has a Yoneda limit), while flat completeness requires every flat weight (a counterpart of ideals in partially ordered sets) has a supremum. It is proved that flat completeness implies Yoneda completeness, but, the converse implication holds only in the case that the related triangular norm is either isomorphic to the {\L}ukasiewicz t-norm or to the product t-norm.
\end{abstract}

\begin{keyword}
Category theory    \sep continuous t-norm \sep quantale \sep  quantaloid \sep ordered fuzzy set \sep forward Cauchy net \sep Yoneda limit \sep Yoneda complete \sep flat weight \sep flat complete
\end{keyword}

\end{frontmatter}

\section{Introduction}
A partially ordered set $P$ is directed complete if each directed subset of $P$ has a supremum. Directed completeness is of fundamental importance in the theory of partial orders \cite{Gierz2003}, so, it is tempting  to establish a counterpart of this notion in the fuzzy setting. In  this introduction, we explain the problems we encounter in this process, summarize briefly what has been done in the literature and what will be done in this paper.

Let $\CQ=(Q,\&,1)$ be a quantale.
There are two kinds of fuzzy orders valued in $\CQ$: the first is ($\CQ$-valued) fuzzy orders on crisp sets; the second is orders on ($\CQ$-valued) fuzzy sets. A crisp set together with a fuzzy order (valued in $\CQ$) is,  from the point of view of category theory,   a category enriched over $\CQ$ (considered as a one-object monoidal biclosed category), with generalized metric spaces in the sense of Lawvere \cite{Lawvere1973} and fuzzy orders in the sense of Zadeh \cite{Zadeh1971}  as prominent examples. A fuzzy set (valued in  $\CQ$) equipped with an order is a category enriched in the quantaloid ${\sf D}(\CQ)$ of diagonals in $\CQ$ \cite{Hoehle2015,Hoehle2011,Pu2012,Stubbe2014},   with sheaves \cite{Betti-Carboni82,Walters81}, $\Omega$-posets \cite{Borceux1998}, and generalized partial metric spaces  \cite{Hoehle2011,Kuenzi2006,Matthews1994,Pu2012} as prototypes. So, both kinds of fuzzy orders are enriched categories,  with the first being over a quantale and the second over a quantaloid. The reader is referred to \cite{Stubbe2014} for a nice introduction to the relationship between fuzzy orders and quantaloid-enriched categories, see also \cite{Hoehle2015,Hoehle2011,Pu2012}.  It should be noted that the category of $\CQ$-categories (i.e., crisp sets together with $\CQ$-valued orders) can be identified with a full subcategory of ${\sf D}(\CQ)$-categories (i.e., ordered fuzzy sets valued in $\CQ$). So, the study of fuzzy orders on crisp sets is a special case of that of orders on fuzzy sets. Fuzzy orders on crisp sets have received wide attention (to name a few, \cite{America1989,Bvelohlavek2004,Bonsangue1998,Hofmann2012,Lai2006,Lai2007,Rutten1998,Wagner1994,Wagner1997,Waszkiewicz2009}), but, the general theory of orders on fuzzy sets is still at its beginning steps, many things remain to be unveiled.

Directed completeness of  $\mathcal{Q}$-categories (i.e.  crisp sets equipped with fuzzy orders) has already received wide attention. The aim of this paper is to investigate directed completeness of ordered fuzzy sets (or, ${\sf D}(\CQ)$-categories) in the case that $\CQ$ is the unit interval $[0,1]$ coupled with a continuous t-norm $\&$. These quantales play an important role in fuzzy set theory, for instance, the BL-logic \cite{Hajek1998} is a logic based on continuous t-norms.

In order to explain what has been done in the literature and what will be done in this paper, we recall two equivalent characterizations of directed completeness first. Let $P$ be a partially ordered set. A subset of $P$ is called an ideal if it is a directed and a lower set. A net $\{x_i\}$ in $P$ is eventually monotone if there is some $i$ such that $x_j\leq x_k$ whenever $i\leq j\leq k$. An element $y\in P$ is an eventual upper bound of a net $\{x_i\}$   if there is some $i$ such that $x_j\leq y$ whenever $i\leq j$.  For a partially ordered set $P$,  it is clear that the following are equivalent: \begin{itemize}\item $P$ is directed complete. \item Each ideal of $P$ has a supremum. \item For each eventually monotone net $\{x_i\}$ in $P$, there is some $x\in P$ such that for all $y\in P$, $x\leq y$ if and only if $y$ is an eventual upper bound of $\{x_i\}$. Said differently, each eventually monotone net has a least eventual upper bound. \end{itemize}

Both the approach of ideals (i.e., each ideal has a supremum) and the approach of nets (i.e., each eventually monotone net has a least eventual upper bound) to directed completeness  have been extended to $\mathcal{Q}$-categories.
For the approach of nets, forward Cauchy nets (a $\CQ$-version of eventually monotone nets) have been introduced, resulting in the notion of Yoneda complete $\mathcal{Q}$-categories (a.k.a. liminf complete $\mathcal{Q}$-categories), see e.g. \cite{Bonsangue1998,Flagg1996,Goubault,Kuenzi2002,Lai2016,Vickers2005,Wagner1997}.
For the approach of ideals, certain classes  of weights in $\mathcal{Q}$-categories have been proposed as $\CQ$-versions of ideals, see e.g. \cite{Flagg1997,Flagg2002,Flagg1996,Kostanek2011,Lai2007,Schmitt2009,Vickers2005,Waszkiewicz2009}. The  approach via ideals is in fact an instance of the  theory of $\Phi$-cocompleteness for enriched categories \cite{Albert1988,Kelly1982,Kelly2005}.
It should be stressed that the situation with $\CQ$-categories is much more complicated than the classic case. To see this, we list two facts here. The first, there lacks a ``standard" choice of weights that can be treated as a counterpart of ideals in partially ordered sets. Instead, different classes of weights have been proposed in the literature for different kinds of quantales.  For instance, $\mathcal{V}$-ideals for $\mathcal{V}$-continuity spaces \cite{Flagg1997,Flagg2002,Flagg1996},  flat weights for generalized metric spaces \cite{Vickers2005}, etc.  The second,  though the two approaches are equivalent in the classic case,  their relationship is not clear in the fuzzy setting, see e.g.  \cite{Hofmann2012,Schmitt2009}. However, there have been some interesting results in this regard. For instance, for $\mathcal{V}$-continuity spaces, Yoneda completeness is equivalent to that each $\CV$-ideal has a supremum \cite{Flagg2002,Flagg1996}; for  generalized metric spaces,  Yoneda completeness is equivalent to that each flat weight has a supremum \cite{Vickers2005}.

In this paper, both approaches  to directed completeness will be extended to ordered fuzzy sets in the case that $\CQ$ is the quantale obtained by endowing the unit interval $[0,1]$ with a continuous t-norm $\&$.   For the approach of ideals,  flat completeness, that requires every flat weight has a supremum, is considered; for the approach of nets,  Yoneda completeness, that requires every forward Cauchy net has a Yoneda limit, is considered. The focus is on the relationship between flat completeness and Yoneda completeness. It is shown that flat  completeness always implies Yoneda completeness, but the converse implication holds  only when the t-norm $\&$ is either isomorphic to the {\L}ukasiewicz t-norm or to the product t-norm. These results exhibit a deep connection between properties of ordered fuzzy sets and the structure of $\CQ$ --- the table of truth-values.

The contents are arranged as follows.
Section 2  recalls some basic ideas about quantales, continuous t-norms, and quantaloids. Section 3  recalls the notion  of ordered fuzzy sets valued in a quantale $\mathcal{Q}$, with emphasis on the fact that they are categories enriched over a quantaloid constructed from $\CQ$.
Section 4 introduces the notions of flat weights  and flat  completeness for ordered fuzzy sets. These notions make sense for categories enriched in any quantaloid.
Section 5 introduces the concepts of forward Cauchy nets and Yoneda completeness for ordered fuzzy sets and presents some of their basic properties.
Section 6  shows that if $\&$ is a continuous t-norm, then flat completeness implies Yoneda completeness for ordered fuzzy sets valued in the quantale $([0,1],\&,1)$.
Section 7 proves that if $\&$ is a continuous t-norm, then   each Yoneda complete ordered fuzzy set (valued in the quantale $([0,1],\&,1)$) with an isolated element (defined below) is flat complete  if and only if $\&$ has no non-trivial idempotent elements. Concluding remarks  are in the last section.

\section{Preliminaries: continuous t-norms, quantales, and quantaloids}
A {\it quantale} $\CQ$ \cite{Rosenthal1990} is a triple $(Q,\&, 1)$,  where $Q$ is a complete lattice, $1$ is an element of $Q$, and $\&$ is a semigroup operation on $Q$ such that $a\&(\bigvee b_i ) =\bv (a\&b_i )$, $(\bigvee b_i)\&a =\bigvee(b_i\&a)$, and $1\&a = a = a\&1$ for all $a, b_i \in Q$. The top element and bottom element of $Q$ will be denoted by $\top$ and $\bot$ respectively. A  homomorphism  $f: (Q,\&, 1)\lra (Q',\&', 1')$ between  quantales \cite{Rosenthal1990} is a map $f: Q\lra Q'$ such that $f(1)=1'$, $f(\bv_{t\in T} a_t)=\bv_{t\in T} f(a_t)$,   and  $f(a\&b)=f(a)\&'f(b)$ for all $a,b,a_t\in Q$. It is clear that quantales and homomorphisms form a category.

Given a quantale $(Q,\&, 1)$ and $a\in Q$, there exist two pairs of adjunctions $-\& a\dashv-\swarrow a$ and $a\&-\dashv a\searrow-$. The binary operations $\swarrow, \searrow$ will be called \emph{the left  and  right implications} in $(Q,\&, 1)$, respectively. For any $a,b,c\in Q$, we have by definition $$ a\leq c\swarrow b \Longleftrightarrow  a\& b\leq c  \Longleftrightarrow  b\leq a\searrow c.$$

A quantale $(Q,\&, 1)$ is \emph{commutative} if $a\&b=b\&a$ for all $a,b\in Q$. If $(Q,\&, 1)$ is commutative, then $a\searrow  b=b\swarrow a$ for all $a,b\in Q$. In this case, we write $a \ra b$ for both $b\swarrow a$ and $a\searrow b$.
A quantale $(Q,\&, 1)$ is {\it divisible} if for any $x, y \in Q$,  $x \leq y$ implies that $ a\&y = x = y\&b$ for some $a, b\in Q$.

\begin{exmp}\label{exmp} \begin{enumerate}[(1)] \item (\cite{Rosenthal1990}) A frame is a complete lattice $H$ such that for each $a\in H$, the operation $a\wedge-$ distributes over arbitrary joins. If $H$ is a frame then $(H,\wedge,1)$ is  a commutative and divisible quantale.
\item  Lawvere's quantale  $([0,\infty]^{\rm op},+,0)$, introduced in \cite{Lawvere1973}, is both  commutative and divisible   with $a\ra b=\max\{0,b-a\}$. \end{enumerate} \end{exmp}

\begin{lem}\label{GL cond}(\cite{RL,PZ2012}) Let $(Q,\&, 1)$ be a quantale. The following conditions are equivalent:
\begin{enumerate}
\item[\rm (1)] $(Q,\&,1)$ is divisible.
\item[\rm (2)] $\forall a,b\in Q$, $a\leq b \Rightarrow a=b\& (b\searrow a)=(a\swarrow b)\& b $.
\item[\rm (3)] $\forall a,b,c\in Q$, $a,b\leq c\Rightarrow a\& (c\searrow b)=(a\swarrow c)\& b$.
\item[\rm (4)] $\forall a,b\in Q$,
$(b\swarrow a)\& a=a\wedge b=a\& (a\searrow b)$.
\end{enumerate}
In this case, the underlying lattice $Q$ is a frame and the unit $1$ must be the top element.
\end{lem}

A continuous t-norm on $[0,1]$ is a binary operation $\&$ on $[0,1]$ that makes $([0,1],\&,1)$ into a  divisible quantale \cite{Hajek1998}. It is well-known that if $\&$ is a continuous t-norm on $[0,1]$ then $\&$ must be commutative \cite{Mostert1957}. 

\begin{exmp}\label{continuity of implication} (\cite{Hajek1998,Klement2000,Mostert1957}) Three basic continuous t-norms:
\begin{enumerate}
\item[(1)] The G\"{o}del t-norm $\&_M$: $a\&_Mb= \min\{a, b\}$. For this t-norm, $a\ra b= 1$ if $a\leq b$ and $a\ra b= b$ if $a> b$.

\item[(2)] The {\L}ukasiewicz t-norm $\&_{\L}$: $a\&_{\L}b=\max\{a+b-1,0\}$. For this t-norm, $a\ra b= \min\{1-a+b,1\}$. It is clear that $\ra$ is continuous on $[0,1]\times[0,1]$ in this case.

\item[(3)]The product t-norm $\&_P$: $a\&_Pb=a\cdot b$. For this t-norm, $a\ra b= 1$ if $a\leq b$ and $a\ra b= b/a$ if $a> b$. In this case, $\ra$ is continuous on $(0,1]\times[0,1]$. It is clear that  $([0,1],\&_P,1)$ is isomorphic to Lawvere's quantale  $([0,\infty]^{\rm op},+,0)$.
\end{enumerate}\end{exmp}

The notion of continuous t-norms makes sense  on any closed interval $[a,b]$ $(a<b)$ in $\mathbb{R}$, the only thing one needs to do is to replace  $0$ and $1$ by $a$ and $b$, respectively.

An element $a\in [0,1]$ is  idempotent with respect to a continuous t-norm $\&$ if $a\&a=a$.   If $a\in [0,1]$ is idempotent, then \begin{equation}\label{tensor with an idempoent} a\&b=\min\{a, b\}\end{equation} for all $b\in[0,1]$ (see e.g. \cite{Klement2000}, Proposition 2.3) and \begin{equation}\label{implication around an idempotent} c\ra b=b\end{equation} whenever $b<a\leq c$.

It is obvious that if $\&$ is a continuous t-norm on $[0,1]$ and $a, b$ ($a<b$) are idempotent with respect to $\&$, then the restriction of $\&$ on $[a,b]$ is a continuous t-norm on $[a,b]$.

By abuse of language, we say that continuous t-norms $\&,\&'$ on $[0,1]$ are isomorphic if the quantales $([0,1],\&,1)$ and $([0,1],\&',1)$ are isomorphic.  The following conclusion is of fundamental importance in the theory of continuous t-norms.

\begin{thm}\label{ord sum} (\cite{Klement2000,Mostert1957}) Let $\&$ be a continuous t-norm on $[0,1]$. If $a\in (0,1)$ is non-idempotent, then there exist idempotent elements $a_{-}, a^{+}\in [0,1]$ such that $a_-<a<a^+$ and that  the quantale  $([a_{-},a^{+}],\&,a^{+})$ is either isomorphic to $([0,1],\&_{\L},1)$ or to $([0,1],\&_{P},1)$ in the category of quantales and homomorphisms. In particular, if $\&$ has no non-trivial idempotent elements, then $\&$ is either isomorphic to the {\L}ukasiewicz t-norm or to the product t-norm. \end{thm}

Now we present some facts of  continuous t-norms for later use.

\begin{lem}\label{t-norm} Let $\&$ be a continuous t-norm on $[0,1]$ and $a,b\in[0,1]$. Then for any $\varepsilon>0$, there exists some $\delta>0$ such that $a\&((a-\delta)\ra b)<\min\{a,b\}+\varepsilon$.\end{lem}
\begin{proof}
Since $a\&-:[0,1]\lra [0,1]$ is uniformly continuous, there exists some $\varepsilon_0>0$ such that $$|x-y|\leq\varepsilon_0\Rightarrow |a\&x -a\&y|<\varepsilon$$ for all $x,y\in[0,1]$. Since $$\bw_{\lambda>0} ((a-\lambda)\ra b)=a\ra b,$$ there is some $\delta>0$ such that $$((a-\delta)\ra b)-(a\ra b)<\varepsilon_0,$$ then $$a\&((a-\delta)\ra b)< a\&(a\ra b)+\varepsilon =\min\{a,b\}+\varepsilon,$$ completing the proof.
\end{proof}

\begin{lem}
Suppose $\&$ is a continuous t-norm on $[0,1]$. If $a\in [0,1]$ is non-idempotent, then there exists some  $h>0$ such that  $$f:[a-h, a+h]\times[a_-,a^+]\lra [0,1], \quad f(x,y)=\min\{x\ra y, a^+\}$$ is continuous, where, $a_-$ denotes the biggest idempotent element of $\&$  that is smaller than $a$, and $a^+$ denotes the least idempotent element  of $\&$ that is bigger than $a$. \end{lem}
\begin{proof}First of all,   the quantale $([a_-,a^+],\&,a^+)$ is either  isomorphic to $([0,1],\&_{\L},1)$ or to  $([0,1],\&_P,1)$ by Theorem \ref{ord sum}. Take $h>0$ with the condition that $[a-h,a+h]\subseteq (a_-,a^+)$. Then $h$ satisfies the requirement. To see this, let $T$ be the t-norm on $[a_-,a^+]$ obtained by restricting $\&$ to $[a_-,a^+]^2$. Then $$f(x,y)=x\stackrel{T}{\ra}y=\bv\{z\in[a_-,a^+]\mid x\&z\leq y\}$$ for all $(x,y)\in[a-h, a+h]\times[a_-,a^+]$. Therefore, $f(x,y)$ is continuous since  $x\stackrel{T}{\ra}y$ is continuous on $[a-h, a+h]\times[a_-,a^+]$ (see Example \ref{continuity of implication}).
\end{proof}

\begin{lem}\label{uniformly continuous}
If $\&$ is a continuous t-norm on $[0,1]$, then for any $a\in (0,1]$ and $\varepsilon>0$, there exists some $\delta>0$ such that $$((a+\delta)\ra(a-\delta))\&(c-\delta)\geq c-\varepsilon$$ for all $c\in [0,a]$. \end{lem}
\begin{proof}
\textbf{Case 1}. $a$ is idempotent. The binary function $-\&-$ is continuous on $[0,1]^{2}$, hence uniformly continuous. Thus, there exists $\delta >0$ such that for all $(x_i,y_i)\in [0,1]^{2}$ $(i=1,2)$, $$|x_1-x_2|, |y_1-y_2|\leq\delta \Rightarrow |x_1\& y_1-x_2\& y_2|<\varepsilon.$$
Then   $$((a+\delta )\ra(a-\delta ))\&(c-\delta )=(a-\delta )\&(c-\delta )\geq a\&c-\varepsilon=c-\varepsilon$$  since $a$ is idempotent.

\textbf{Case 2}. $a$ is non-idempotent. Let $a_-, a^+$ be assumed as in the above lemma. Then for any $h>0$ with the condition that $[a-h,a+h]\subseteq (a_-,a^+)$, the function  $$ [a-h, a+h]\times[a_-,a^+]\lra [0,1], \quad  (x,y)\mapsto\min\{x\ra y, a^+\}$$ is continuous. Therefore, the ternary function $$ [a-h, a+h]\times[a_-,a^+]\times[0,a^+]\lra [0,1], \quad (x,y,z)\mapsto\min\{x\ra y, a^+\}\&z$$ is continuous, hence uniformly continuous. Since $a^+\&z=\min\{a^+,z\}=z$ for all $z\leq a^+$   by virtue of Equation (\ref{tensor with an idempoent}), it follows that $$\min\{x\ra y, a^+\}\&z=\min\{(x\ra y)\&z, a^+\&z\}=\min\{(x\ra y)\&z, z\}=(x\ra y)\&z,$$  hence  there is some $0<\delta< h$ such that $$|(x_1\ra y_1)\&z_1-(x_2\ra y_2)\&z_2|<\varepsilon$$ for any $(x_i,y_i,z_i)\in [a-h, a+h]\times[a_-,a^+]\times[0,a^+]$ $(i=1,2)$ with
 $|x_1-x_2|, |y_1-y_2|, |z_1-z_2|\leq\delta$.
  In particular,  $$((a+\delta)\ra(a-\delta))\&(c-\delta)\geq (a\ra a)\&c-\varepsilon=c-\varepsilon $$ for all $c\in [0,a]$.
\end{proof}

There is still another notion that will be needed in this paper, that of quantaloid-enriched categories. A
quantaloid is to a quantale what a groupoid is to a group. Precisely,   a  quantaloid \cite{Rosenthal1996} is a category $\mathcal{Q}$ such that the set $\mathcal{Q}(a,b)$ of the arrows from $a$ to $b$ is a complete lattice for all objects $a,b$ in $\mathcal{Q}$; and that the composition $\circ$ preserves suprema in both variables, i.e.,
$$  \alpha \circ \bigvee_{i \in I} \beta_i =\bigvee_{i \in I}(\alpha \circ \beta_i ), \ \ \Big(\bigvee_{j \in J} \alpha_j\Big)\circ \beta =\bigvee_{j \in J}(\alpha_j\circ \beta)$$
for  all  $\alpha, \alpha_{j} \in \mathcal{Q}(b,c)$ and  $\beta, \beta_i \in  \mathcal{Q}(a,b).$
The bottom and top element  of $\mathcal{Q}(a,b)$ are denoted by $\bot_{a,b}$ and $\top_{a,b}$, respectively;   the identity arrow on an object $a$ is denoted by $1_{a}$.

Given a quantaloid $\CQ$ and $\CQ$-arrows $\alpha:b\lra c$ and $\beta:a\lra b$, there are two adjunctions
\begin{align*}
-\circ \beta\dashv -\swarrow \beta:&\ \CQ(b,c)\lra \CQ(a,c),\\
\alpha\circ -\dashv \alpha\searrow -:&\ \CQ(a,b)\lra \CQ(a,c)
\end{align*} determined by the adjoint property
$$\alpha\leq \gamma\swarrow \beta \Longleftrightarrow  \alpha\circ\beta\leq \gamma  \Longleftrightarrow  \beta\leq \alpha\searrow \gamma.$$
  The right adjoints $\swarrow,\searrow$ will   be called the \emph{left} and \emph{right implications}, respectively.

Suppose that
$\mathcal{Q}$ is a quantaloid. A $\mathcal{Q}$-category\footnote{The terminologies  adopted here are not exactly  the same as in our main reference,  Stubbe \cite{Stubbe2005}, on quantaloid-enriched categories. Our $\CQ$-categories and $\CQ$-distributors are exactly the $\CQ^{\rm op}$-categories and   $\CQ^{\rm op}$-distributors in the sense of Stubbe, where $\CQ^{\rm op}$ is the quantaloid obtained by reversing the 1-morphisms in $\CQ$. The difference arises in the interpretation of $\mathbb{A}(x,y)$: it is interpreted as the hom-arrow from $y$ to $x$ in \cite{Stubbe2005}, but from $x$ to $y$ here. Our terminologies agree with that in \cite{Rosenthal1996,Walters81}.} \cite{Stubbe2005}
 $\mathbb{A}$ consists of the following data: \begin{itemize} \item a set $\mathbb{A}_0$ of objects; \item    a map
$t$ from $\mathbb{A}_0$ to the set $\CQ_0$ of objects in $\mathcal{Q}$, which is called the type function;   \item an element $\mathbb{A}(x,y)\in \mathcal{Q}(tx,ty)$ for each pair $(x,y)$ in $\bbA_0$. \end{itemize} These data are required to satisfy the following conditions:
\begin{enumerate}\item[(1)] $1_{tx}\leq \mathbb{A}(x,x)$ for all $x\in \mathbb{A}_0$;
\item[(2)] $\mathbb{A}(y,z)\circ\mathbb{A}(x,y)\leq\mathbb{A}(x,z)$ for all $x,y,z\in \mathbb{A}_0$.
\end{enumerate}

A quantale is exactly a quantaloid with only one object. Interestingly, given a quantale $\CQ=(Q,\&,1)$, one can construct a quantaloid ${\sf D}(\CQ)$, called the quantaloid of diagonals in $\CQ$,\footnote{This construction makes sense for any quantaloid, not only for those with one object, see Stubbe \cite{Stubbe2014}.}  as follows:
\begin{itemize}
\item objects: elements $a,b,c,...,$ in $Q$;
\item  morphisms: ${\sf D}(\CQ)(a,b)=\{\alpha\in Q \mid (\alpha\swarrow  a)\& a=\alpha=b\&(b\searrow  \alpha)\}$;
\item composition: $\beta \circ \alpha=\beta\&(b\searrow  \alpha)=(\beta\swarrow  b)\& \alpha$ for all
$\alpha\in {\sf D}(\CQ)(a,b)$ and $\beta\in {\sf D}(\CQ)(b,c)$;
\item the unit $1_a$ of ${\sf D}(\CQ)(a,a) $ is $a$;
\item the partial order on ${\sf D}(\CQ)(a,b)$ is inherited from $Q$.
\end{itemize}

It is easy to see that if $\CQ$ is divisible, then ${\sf D}(\CQ)(a,b)=\{\alpha\in Q\mid \alpha\leq a\wedge b\}$.

A quantale $\CQ=(Q,\&,1)$, considered as a one-object category, can be treated as a full subcategory of ${\sf D}(\CQ)$ by identifying the only object of $\CQ$ with the object $1$ in ${\sf D}(\CQ)$. Therefore, $\CQ$-categories in the sense of \cite{Wagner1997}, which are sets endowed with fuzzy orders from the viewpoint of fuzzy set theory,  are a special kind of ${\sf D}(\CQ)$-categories.

\section{Ordered fuzzy sets valued in a quantale}
Given a quantale $\CQ=(Q,\&,1)$, following \cite{Hoehle2015,Hoehle2011,Pu2012,Tao2014}, an ordered fuzzy set valued in $\CQ$ is defined to be a category enriched over the quantaloid ${\sf D}(\CQ)$ of diagonals in $\CQ$.\footnote{It should be warned that our definition of ordered fuzzy sets are not exactly the same as that in H\"{o}hle \cite{Hoehle2015,Hoehle2011}. The difference arises in the definition of quantaloid-enriched categories. Since the papers \cite{Hoehle2015,Hoehle2011}  follow the terminologies   of Stubbe \cite{Stubbe2005}, an ordered fuzzy set defined there is exactly  a  ${\sf D}(\CQ)^{\rm op}$-category  (not a ${\sf D}(\CQ)$-category) in our sense. We note in pass that for each quantale $\CQ=(Q,\&,1)$, the quantaloid ${\sf D}(\CQ)^{\rm op}$ is isomorphic to ${\sf D}(\CQ^{\rm op})$, where $\CQ^{\rm op}$ refers to  the quantale $(Q,\&^{\rm op}, 1)$, where $a\&^{\rm op}b=b\&a$ for all $a,b\in Q$.}  Explicitly,

\begin{defn}\label{ofs} Let $\CQ=(Q,\&,1)$ be a quantale. An ordered fuzzy set valued in $\CQ$ consists of the following data: \begin{itemize}
\item a ($Q$-valued) fuzzy set $(\bbA_0,t)$, i.e., a set $\bbA_0$ and a  membership  function $t:\bbA_0\lra Q$; \item an element $\bbA(x,y)\in {\sf D}(\CQ)(tx,ty)$ for any pair $(x,y)$ of elements in $\bbA_0$,   measuring the degree that $x$ precedes $y$.
    \end{itemize} These data are required to satisfy the following conditions: for all $x, y, z\in\bbA_0$,
\begin{enumerate}[(1)]
\item  $tx\leq \bbA(x,x)$;
\item $((\bbA(y,z)\swarrow ty))\&\bbA(x,y) \leq\bbA(x,z)$. \end{enumerate} \end{defn}

The conditions (1)   and (2) express respectively reflexivity and transitivity. Following the terminologies of quantaloid-enriched categories, the membership function $t$ will  also be called the type function, and $tx$   the type (instead of membership degree) of $x$. We often write $\bbA$,  instead of $(\bbA_0,t,\bbA)$, for an ordered fuzzy set.

The following examples provide us two important kinds of ordered fuzzy sets.

\begin{exmp}\label{PMS}Let $\CQ$ be Lawvere's quantale  $([0,\infty]^{\rm op},+,0)$. Then a ${\sf D}(\CQ)$-category is essentially a pair $(X,p)$, where,  $X$ is a set  and
$p:X\times X\lra [0,+\infty]$ is a map such that for
all $x,y,z\in X$,
\begin{enumerate}
\item[(1)] $p(x,x),p(y,y)\leq p(x,y)$;
\item[(2)] $p(y,z)+ (p(x,y)-p(y,y))\geq p(x,z)$.
\end{enumerate} Here we take by convention that $\infty-\infty=0$.
These enriched categories  are called generalized partial metric spaces in \cite{Pu2012} because they extend the notion of partial metric spaces in \cite{Kuenzi2006,Matthews1994}. \end{exmp}

\begin{exmp} Let $H$ be a frame. Then  $\Omega=(H,\wedge,1)$ is a divisible quantale,   a ${\sf D}(\Omega)$-category  is  exactly an $\Omega$-poset in the sense of \cite{Borceux1998}. Precisely, a ${\sf D}(\Omega)$-category is  a set $\bbA_0$  together with a map $\bbA: \bbA_0\times \bbA_0\lra H$   such that for
all $x,y,z\in \bbA_0$,
\begin{enumerate}
\item[(1)] $\bbA(x,y)\leq \bbA(x,x)\wedge\bbA(y,y)$;
\item[(2)] $\bbA(y,z)\wedge\bbA(x,y)\leq \bbA(x,z)$.
\end{enumerate} Moreover, Cauchy complete and symmetric ${\sf D}(\Omega)$-categories are exactly the sheaves on $H$ \cite{Betti-Carboni82,Walters81}.
\end{exmp}

Let $\CQ=(Q,\&,1)$ be a quantale.   An element $x$ in an ordered $\CQ$-fuzzy set   $\bbA$ is \emph{global} if $tx=\bbA(x,x)\geq1$. An ordered $\CQ$-fuzzy set $\bbA$ is \emph{global} if all of its elements are global. It is clear that if $\CQ$ is integral, i.e., the unit $1$  of $\CQ$ is the top element, then $\CQ$-categories in the sense of \cite{Wagner1997} are precisely those global ${\sf D}(\CQ)$-categories.

In the rest of this paper, we are mainly concerned with ordered fuzzy sets valued in the quantale obtained by endowing the interval $[0,1]$ with a continuous t-norm, so, we write down the details of such ordered fuzzy sets (in a slightly different way) for later use.

Let $\&$ be a continuous t-norm on $[0,1]$. An order on a fuzzy set $(\bbA_0, t)$ (valued in the quantale $\CQ=([0,1],\&,1)$) is a map $\bbA:\bbA_0\times\bbA_0\lra [0,1]$ such that for all $x, y, z\in\bbA_0$,
\begin{enumerate}[(i)]
\item  $\bbA(x,y)\leq tx\wedge ty$;
\item  $tx\leq \bbA(x,x)$;
\item $(ty\ra\bbA(y,z))\&\bbA(x,y) \leq\bbA(x,z)$.
\end{enumerate}

Condition (i) is  the requirement   $\bbA(x,y)\in {\sf D}(\CQ)(tx,ty)$ in Definition \ref{ofs}, since ${\sf D}(\CQ)(tx,ty)=\{\alpha\in [0,1]\mid \alpha\leq tx\wedge ty\}$ in this case.  It is easily inferred from (i) and (ii) that $tx=\bbA(x,x)$ for all $x\in\bbA_0$.

\section{Flat and Cauchy completeness}\label{section 4}
Flat and Cauchy completeness make sense for any quantaloid-enriched category. To state the definition, we need some preparations.

In this section, $\CQ$ always denotes a quantaloid unless otherwise specified. A $\mathcal{Q}$-distributor $\phi:\mathbb{A}\oto \mathbb{B}$ between $\CQ$-categories is a map $\phi$ from $\mathbb{A}_0\times\mathbb{B}_0$ to the set of morphisms in $\mathcal{Q}$ subject to the following conditions:
\begin{enumerate}[(1)] \item $\forall  x\in \mathbb{A}_0,\forall y\in \mathbb{B}_0$,
$\phi(x,y)\in\mathcal{Q}(tx,ty)$;
\item $\forall x\in \mathbb{A}_0,\forall y,y'\in \mathbb{B}_0$;
$\mathbb{B}(y',y)\circ\phi(x,y')\leq\phi(x,y)$;
\item  $\forall x,x'\in \mathbb{A}_0,\forall y\in \mathbb{B}_0$, $\phi(x',y)\circ
\mathbb{A}(x,x')\leq\phi(x,y)$. \end{enumerate}

Given $\mathcal{Q}$-distributors
$\mathbb{A}\stackrel{\phi}{\oto}\mathbb{B}$ and $\bbB\stackrel{\psi}{\oto}\bbC$, the
composite $ \psi\circ\phi:\mathbb{A}{\oto}\bbC$ is given by
$$\forall x\in \mathbb{A}_0,\forall z\in \bbC_0,\psi\circ\phi(x,z)
=\bv_{y\in \mathbb{B}_0}\psi(y,z)\circ\phi(x,y).$$ $\mathcal{Q}$-categories and
distributors constitute a category $\mathcal{Q}$-{\bf Dist},  it is moreover
a quantaloid \cite{Stubbe2005}. In particular, $\mathcal{Q}$-{\bf Dist}$(\mathbb{A},\mathbb{B})$ (under pointwise order inherited from $\mathcal{Q}$) is a complete lattice for any $\mathcal{Q}$-categories $\mathbb{A}$ and $\mathbb{B}$. The identity arrow of $\mathcal{Q}$-{\bf Dist}$(\mathbb{A},\bbA)$ in $\mathcal{Q}$-{\bf Dist} is given by   $\bbA: \bbA\oto \bbA$. For $\CQ$-distributors $\phi:\mathbb{A}\oto\mathbb{B}$, $\psi:\mathbb{B}\oto\mathbb{C}$ and $\eta:\mathbb{A}\oto\mathbb{C}$,  the left
implication $\eta\swarrow\phi: \mathbb{B}\oto\mathbb{C}$ and right implication $\psi\searrow\eta:\mathbb{A}\oto\mathbb{B}$
are calculated as follows:
\begin{align*}\forall y\in\mathbb{B}_0, z\in\mathbb{C}_0, (\eta\swarrow\phi)(y,z)=\bw_{x\in\bbA_0}\eta(x,z)\swarrow \phi(x,y);\\ \forall x\in\mathbb{A}_0, y\in\mathbb{B}_0, (\psi\searrow\eta)(x,y)=\bw_{z\in\mathbb{C}_0}\psi(y,z) \searrow\eta(x,z).\end{align*}

A pair of $\mathcal{Q}$-distributors $\psi:\mathbb{A}\oto\mathbb{B}$ and $\phi: \mathbb{B}\oto\mathbb{A}$ forms an adjunction in $\mathcal{Q}$-{\bf Dist} if $$\bbA\leq\phi\circ\psi \quad {\rm and} \quad \psi\circ\phi\leq\bbB.$$
In this case, $\psi$ is called a left adjoint of $\phi$ and $\phi$ a right adjoint of $\psi$.

For each object $a$ of a quantaloid $\mathcal{Q}$, write $*_a$ for the $\mathcal{Q}$-category with exactly one object, say $*$,  such that  $t*=a$ and ${\rm hom}(*,*)=1_a$. It is clear that $\mathcal{Q}$-${\bf Dist}(*_a,*_b)$ is essentially the complete lattice $\CQ(a,b)$, so we won't distinguish   $\CQ(a,b)$ and $\mathcal{Q}$-${\bf Dist}(*_a,*_b)$ in the sequel.

A {\it weight} (or, a presheaf) with type $t\phi$ on a $\mathcal{Q}$-category $\mathbb{A}$ is a $\mathcal{Q}$-distributor $\phi:\mathbb{A}\oto *_{t\phi}$. For a weight $\phi:\mathbb{A}\oto*_{t\phi}$, we often write $\phi(x)$ instead of $\phi(x,*)$ for short. All weights on $\mathbb{A}$ form a $\mathcal{Q}$-category $\mathcal{P}\mathbb{A}$ with $$ \CP\mathbb{A}(\phi,\varphi)=\varphi\swarrow\phi=\bw_{x\in\mathbb{A}_0}\varphi(x)\swa\phi(x).$$

For each object $x$ in a $\mathcal{Q}$-category $\mathbb{A}$, there is a distributor  $\y(x):\mathbb{A}\oto *_{tx}$  given by $\y(x)(y)=\mathbb{A}(y,x)$.

\begin{lem}[Yoneda lemma]\label{Yoneda} (\cite{Stubbe2005}) For all $x\in\mathbb{A}_0$ and $\phi\in\CP\mathbb{A}$, $\CP\mathbb{A}(\y(x),\phi)=\phi(x)$.\end{lem}

Dually, a {\it coweight} (or, a co-presheaf) with type $t\psi$ on $\mathbb{A}$ is a distributor $\psi:*_{t\psi}\oto\mathbb{A}$.

\begin{defn} (1) (\cite{Stubbe2005}) A weight $\phi:\mathbb{A}\oto*_{t\phi}$ is Cauchy if it is a right adjoint.

(2) (\cite{Tao2014})
A weight $\phi:\mathbb{A}\oto*_{t\phi}$ is flat if for each object $a$ in $\mathcal{Q}$, the map $$\mbox{$\phi\circ-:\mathcal{Q}${\rm -\bf Dist}$(*_a,\mathbb{A})\lra \mathcal{Q}(a,t\phi)$ }$$ preserves arbitrary finite meets (in particular, $\phi\circ-$ preserves the top element.). \end{defn}

\begin{prop}\label{Cauchy is flat}  Each Cauchy weight on any $\CQ$-category is flat.\end{prop}

\begin{proof} Suppose that $\phi:\bbA\oto*_a$ is a Cauchy weight on a $\CQ$-category $\bbA$ with a left adjoint given by $\psi:*_a\oto \bbA$.  Then  $\phi\circ-=\psi\searrow-$ (see e.g. \cite{Heymans2010}, a direct verification is also easy), hence $\phi\circ-$ preserves arbitrary meets, showing that $\phi$ is flat.
\end{proof}

\begin{defn}  A supremum of a weight $\phi:\mathbb{A} \oto *_{t\phi}$  is an element $x\in\bbA_0$ of type $t\phi$ such that $\bbA(x,y)=\CP\mathbb{A}(\phi,\y(y))$ for all $y\in\bbA_0$.\end{defn}

Suprema are special cases of weighted colimits in $\CQ$-categories \cite{Stubbe2005}. Actually,  a supremum of a weight $\phi:\mathbb{A} \oto *_{t\phi}$ is exactly the colimit of the identity functor $\bbA \lra\bbA$ weighted by the distributor $\phi:\mathbb{A} \oto *_{t\phi}$.

\begin{defn}Let  $\mathbb{A}$ be a $\mathcal{Q}$-category.
\begin{enumerate}
\item [\rm(1)] $\mathbb{A}$ is flat complete if each flat weight on $\mathbb{A}$  has a supremum.
\item [\rm(2)] $\mathbb{A}$ is Cauchy complete if each Cauchy weight on $\mathbb{A}$  has a supremum.
\end{enumerate}\end{defn}

The notion of Cauchy completeness dates back to the seminal work of Lawvere \cite{Lawvere1973}, where distributors are called bimodules. For more on Cauchy complete $\CQ$-categories and Cauchy completions the reader is referred to Stubbe \cite{Stubbe2005} and the bibliography therein.

It follows from Proposition \ref{Cauchy is flat} that flat completeness implies Cauchy completeness. In the terminology of category theory, what we call flat completeness should be called flat cocompleteness. However, following the tradition of order theory \cite{Gierz2003},  we choose the term \emph{flat completeness} here.

Since almost all results of this paper are proved for categories enriched in the quantaloid  ${\sf D}(\CQ)$, where $\CQ$ is the divisible quantale obtained by endowing the interval  $[0,1]$ with a continuous t-norm, we list here, for later use, some facts about ${\sf D}(\CQ)$-categories for a commutative  and divisible quantale  $\CQ=(Q,\&,1)$.

\textbf{Fact 1}. For any  weights $\phi$, $\varphi$ and coweight $\psi$ on a ${\sf D}(\CQ)$-category $\bbA$, it holds that  $$\phi\circ\psi=\bv_{x\in\bbA_0}(\bbA(x,x)\ra\phi(x)) \&\psi(x) =\bv_{x\in\bbA_0}\phi(x) \&(\bbA(x,x)\ra\psi(x))$$
and \begin{equation} \label{subsethood} \CP\mathbb{A}(\phi,\varphi)= t\varphi\wedge\bigwedge_{x\in\bbA_0}(\phi(x)\ra \varphi(x))\&t\phi.\end{equation}

\textbf{Fact 2}. For any  ${\sf D}(\CQ)$-category $\bbA$ and   $b\in Q$, the top element of $\CQ${\rm -\bf Dist}$(*_b,\mathbb{A})$ is $b\wedge t$ that maps $x\in\bbA_0$ to $b\wedge tx=b\wedge\bbA(x,x).$ Hence, a weight $\phi:\bbA\oto *_a$ is flat if and only if
\begin{itemize}
  \item   $\phi\circ(b\wedge t)=b\wedge a$ for all $b\in Q$;
  \item  $\phi\circ(\psi_1\wedge\psi_2) =(\phi\circ\psi_1)\wedge(\phi\circ\psi_2)$ for any coweights $\psi_1,\psi_2$ with $t\psi_1=t\psi_2$.
\end{itemize}

\textbf{Fact 3}.
For any ${\sf D}(\CQ)$-category $\mathbb{A}$, $$\phi:\bbA\oto*_\bot, \quad x\mapsto \bot$$  is a weight (of type $\bot$) on $\bbA$. This weight is called the \emph{trivial weight} on $\bbA$. The trivial weight is clearly  Cauchy, hence flat. Thus, every Cauchy complete  ${\sf D}(\CQ)$-category must have an  element of type $\bot$.  Elements of type $\bot$ in a ${\sf D}(\CQ)$-category are said to be \emph{isolated}.

\section{Yoneda completeness and bicompleteness}
Cauchy completeness and flat completeness make sense for categories enriched in any quantaloid. In this section, we introduce two kinds of completeness, Yoneda completeness and bicompleteness, that make sense for categories enriched in quantaloids of the form ${\sf D}(\CQ)$  with $\CQ=(Q,\&,1)$ being a quantale.

Yoneda completeness (a.k.a. liminf completeness), based on the notion of forward Cauchy nets,  was introduced in  \cite{Bonsangue1998,Wagner1994,Wagner1997} as an extension of directed completeness in the realm of quantale-enriched categories (i.e., sets with fuzzy orders).    Yoneda complete generalized metric spaces (i.e.,  categories enriched over  Lawvere's quantale $([0,\infty]^{\rm op},+,0)$) have been studied extensively in the literature, see e.g. \cite{Bonsangue1998,Goubault,Kuenzi2002,Vickers2005}. It is hard to extend this notion to categories enriched over an arbitrary quantaloid. But, as we shall see, for each quantale $\CQ=(Q,\&,1)$, there does exist a natural way to postulate  forward Cauchy nets, hence Yoneda completeness, in   ${\sf D}(\CQ)$-categories (i.e., ordered fuzzy sets).

Recall that a  net $\{a_\lam\}$ in a complete lattice $L$ is \emph{order convergent to $a\in L$} \cite{Birkhoff} if \[ \liminf a_\lam =
\bv_\lam\bw_{\mu\geq \lam}a_\mu = a= \bw_\lam\bv_{\mu\geq \lam}a_\mu =\limsup a_\lam.\]  In this case, $a$ is called the order-limit of $\{a_\lam\}$. It is known that the order convergence in a completely distributive lattice  coincides with that determined by its interval topology (\cite{Erne1980}, Corollary 10).

\begin{defn}\label{FCN}
Let   $\{x_\lam\}_{\lam\in\Lambda}$ be a net in a ${\sf D}(\CQ)$-category $\bbA$.  \begin{enumerate}[\rm(1)]
\item  $\{x_\lam\}_{\lam\in\Lambda}$ has a type if the net $\{\bbA(x_\lam,x_\lam)\}_{\lam\in\Lambda}$ in $Q$ is order convergent, i.e., $$\bigvee\limits_\lam\bigwedge\limits_{\mu\geq\lam}\bbA(x_\mu,x_\mu) =\bigwedge\limits_\lam\bigvee\limits_{\mu \geq\lam}\bbA(x_\mu,x_\mu).$$  The order-limit of $\{\bbA(x_\lam,x_\lam)\}_{\lam\in\Lambda}$ is called the type of the net $\{x_\lam\}_{\lam\in\Lambda}$.
\item $\{x_\lam\}_{\lam\in\Lambda}$ is forward Cauchy if the net $\{\bbA(x_\lam,x_\mu)\mid (\lam,\mu)\in\Lambda\times\Lambda,\lam\leq\mu\}$   is order convergent, i.e., $$\bigvee\limits_\lam\bigwedge\limits_{\nu\geq\mu\geq\lam}\bbA(x_\mu,x_\nu) =\bigwedge\limits_\lam\bigvee\limits_{\nu\geq\mu \geq\lam}\bbA(x_\mu,x_\nu).$$
\item  $\{x_\lam\}_{\lam\in\Lambda}$ is   biCauchy if the net $\{\bbA(x_\lam,x_\mu)\mid (\lam,\mu)\in\Lambda\times\Lambda\}$   is order convergent, i.e., $$\bigvee\limits_\lam\bigwedge\limits_{\nu, \mu\geq\lam}\bbA(x_\mu,x_\nu)= \bigwedge\limits_\lam\bigvee\limits_{\nu, \mu\geq\lam}\bbA(x_\mu,x_\nu).$$
\end{enumerate}\end{defn}

For a net $\{x_\lam\}_{\lam\in\Lambda}$  in a ${\sf D}(\CQ)$-category $\bbA$, one always has that \begin{align*}& ~ \bigvee\limits_\lam\bigwedge\limits_{\nu, \mu\geq\lam}\bbA(x_\mu,x_\nu)  \leq   \bigvee\limits_\lam\bigwedge\limits_{\nu\geq\mu \geq\lam}\bbA(x_\mu,x_\nu) \leq \bigvee\limits_\lam\bigwedge\limits_{\mu\geq\lam}\bbA(x_\mu, x_\mu) \\   \leq & ~\bigwedge\limits_\lam\bigvee\limits_{\mu \geq\lam}\bbA(x_\mu,x_\mu)  \leq \bigwedge\limits_\lam\bigvee\limits_{\nu\geq\mu \geq\lam}\bbA(x_\mu,x_\nu) \leq \bigwedge\limits_\lam\bigvee\limits_{\nu, \mu\geq\lam}\bbA(x_\mu,x_\nu),\end{align*}
it follows that a biCauchy net is forward Cauchy  and a forward Cauchy net has a type.
In particular, if $\{x_\lam\}$ is   forward Cauchy,  then
\begin{equation} \label{eq1} \bigvee_\lam\bigwedge_{\nu\geq\mu\geq\lam}\bbA(x_\mu, x_\nu)  =\bigwedge_\lam\bigvee_{\mu\geq\lam}\bbA(x_\mu,x_\mu); \end{equation}
if   $\{x_\lam\}$ is biCauchy  then
\begin{equation} \label{eq2} \bigvee_\lam\bigwedge_{\nu,\mu\geq\lam}\bbA(x_\mu,x_\nu)
 =\bigwedge_\lam\bigvee_{\mu\geq\lam}\bbA(x_\mu,x_\mu).  \end{equation}

\begin{rem}\label{FC=convergent} (1) If   the underlying complete lattice $Q$ of the quantale $\CQ$ is completely distributive, then the order convergence coincides with the convergence of the interval topology on $Q$ (see \cite{Erne1980}, Corollary 10), hence a net $\{x_\lam\}$ in a ${\sf D}(\CQ)$-category $\bbA$ is forward Cauchy (biCauchy, resp.) if and only if the net $\{\bbA(x_\lam,x_\mu)\mid (\lam,\mu)\in\Lambda\times\Lambda,\lam\leq\mu\}$ ($\{\bbA(x_\lam,x_\mu)\mid (\lam,\mu)\in\Lambda\times\Lambda\}$, resp.) converges  with respect to the interval topology on $Q$.
In particular, if $Q=[0,1]$, then a net $\{x_\lam\}$ in a ${\sf D}(\CQ)$-category $\bbA$ is forward Cauchy (biCauchy, resp.) if and only if the limit $\lim_{\nu\geq\mu}\bbA(x_\mu,x_\nu)$ ($\lim_{\nu,\mu}\bbA(x_\mu,x_\nu)$, resp.) exists (with respect to the standard  topology).

(2) A net $\{x_\lam\}_{\lam\in\Lambda}$ in a ${\sf D}(\CQ)$-category $\bbA$ is \emph{backward Cauchy} if   the net $\{\bbA(x_\lam,x_\mu)\mid (\lam,\mu)\in\Lambda\times\Lambda,\lam\geq\mu\}$   is order convergent. It is clear that a biCauchy net is both forward and backward Cauchy. Conversely, if the order convergence on $Q$ is topological, then a both forward and backward Cauchy net is biCauchy. So, it can be said that biCauchy is a ``symmetric" version of forward Cauchy.

 (3) Let $\CQ=(Q,\&,1)$ be an integral quantale (i.e., the unit $1$ is the top element in $Q$) and let  $\bbA$ be a $\CQ$-category, considered as a (global) ${\sf D}(\CQ)$-category. A net $\{x_\lam\}$ in $\bbA$ is forward Cauchy  if and only if $$\bigvee_\lam\bigwedge_{\nu\geq \mu\geq\lam}\bbA(x_\mu,x_\nu)=1;$$     $\{x_\lam\}$   is biCauchy if and only if $$\bigvee_\lam\bigwedge_{\nu, \mu\geq\lam}\bbA(x_\mu,x_\nu)=1.$$ So, for an integral quantale $\CQ$, the notion of forward Cauchy nets in  ${\sf D}(\CQ)$-categories extends that in $\CQ$-categories in the sense of \cite{Flagg1996,Hofmann2012,Wagner1997};  the notion of biCauchy nets in  ${\sf D}(\CQ)$-categories   extends that of Cauchy nets in  $\CQ$-categories  in the sense of  \cite{Hofmann2013,Hofmann2012,Kuenzi2002}. \end{rem}

\begin{exmp}\label{example of bicauchy net} Let $\CQ=(Q, \&, 1)$ be a divisible quantale and $A$ be a subset of $Q$. Define a ${\sf D}(\CQ)$-category $\bbA$ as follows: the objects of $\bbA$ are the elements in $A$;   for any $x,y\in A$, $\bbA(x,y)=x\wedge y$. Making use of the fact that the underlying lattice $Q$ of $\CQ$ is a frame (see Lemma \ref{GL cond}), it is easily verified that a  net $\{x_\lam\}$ in $\bbA$ is forward Cauchy if and only if it is biCauchy if and only if $\{x_\lam\}$ is order convergent in the lattice $Q$. \end{exmp}

\begin{defn}\label{Yoneda limit}
Let   $\bbA$  be a ${\sf D}(\CQ)$-category, $x\in \bbA_0$, and  $\{x_\lam\}$ a net in $\bbA$.
\begin{enumerate}
\item [\rm(1)] $x$ is   a Yoneda limit of   $\{x_\lam\}$ if $$ \bbA(x,x)= \bigvee_\lam\bigwedge_{\sigma\geq\lam}\bbA(x_\sigma, x_\sigma)$$ and
$$\bbA(x,y)= \bigvee_\lam\bigwedge_{\sigma\geq\lam} \bbA(x_\sigma,y)  $$  for all $y\in\bbA_0$.
\item [\rm(2)] $x$ is   a bilimit of  $\{x_\lam\}$   if
$$\bbA(x,x)=\bigvee_\lam\bigwedge_{\sigma\geq\lam}\bbA(x_\sigma,x)=\bigvee_\lam\bigwedge_{\sigma\geq\lam}\bbA(x,x_\sigma)
=\bigvee_\lam\bigwedge_{\si\geq\lam}\bbA(x_\si,x_\si).$$
\end{enumerate}
\end{defn}

It is clear  that if a net $\{x_\lam\}$ has a type and if $x$ is a Yoneda limit or a bilimit of $\{x_\lam\}$, then $\bbA(x,x)$ is the type of the net $\{x_\lam\}$.

\begin{defn} Let   $\bbA$ be a ${\sf D}(\CQ)$-category.
\begin{enumerate}
\item [\rm(1)] $\bbA$ is  Yoneda complete if every forward Cauchy net in $\bbA$ has a Yoneda limit.
\item [\rm(2)] $\bbA$ is  bicomplete if every biCauchy net in $\bbA$ has a bilimit.
\end{enumerate}\end{defn}

\begin{exmp}Let $\CQ$ and $\bbA$ be assumed as in Example \ref{example of bicauchy net}. Then $\bbA$ is Yoneda complete if and only if $\bbA$ is bicomplete if and only if $A\subseteq Q$ is closed with respect to limits of order convergence. In particular, if $Q$ is completely distributive, then $\bbA$ is Yoneda complete if and only if $A$ is a closed subset of $Q$ in the interval topology. \end{exmp}

The following proposition shows that Yoneda completeness implies bicompleteness if $\CQ$ is the interval $[0,1]$ equipped with a continuous t-norm. But, we do not know whether this is true in the general case.

\begin{prop}\label{LC implies bi-CC} Let $\CQ$ be the interval $[0,1]$ coupled with a continuous t-norm $\&$, $\{x_\lam\}$   a biCauchy net in a ${\sf D}(\CQ)$-category $\bbA$. Then $x\in \bbA_0$ is a Yoneda limit of $\{x_\lam\}$ if and only if $x$ is a bilimit of $\{x_\lam\}$. \end{prop}
\begin{proof}
\textbf{Necessity}. By definition, we only need check that $\bbA(x,x)=\bigvee_\lam\bigwedge_{\sigma\geq\lam}\bbA(x,x_\sigma)$. Since $\bbA(x,x_\sigma)\leq\bbA(x,x)$ for all $x_\sigma$, the inequality $\bigvee_\lam\bigwedge_{\sigma\geq\lam}\bbA(x,x_\sigma) \leq\bbA(x,x)$ is trivial. Conversely,
\begin{align*}
\bv_\lam\bw_{\sigma\geq\lam}\bbA(x,x_\sigma)&=\bv_\lam\bw_{\sigma\geq\lam}\bv_\tau\bw_{\gamma\geq\tau}\bbA(x_\gamma,x_\sigma)& (x\ \text{is a Yoneda limit})\\
&\geq \bv_\lam\bw_{\sigma\geq\lam}\bw_{\gamma\geq\lam}\bbA(x_\gamma,x_\sigma)\\
&=\bv_\lam\bw_{\gamma,\sigma\geq\lam}\bbA(x_\gamma,x_\sigma)\\
&=\bigvee_\lam\bigwedge_{\si\geq\lam}\bbA(x_\si,x_\si)&(\{x_\lam\}\ \text{is biCauchy})\\
&=\bbA(x,x).\end{align*}

\textbf{Sufficiency}. It suffices to check that $\bbA(x,y)= \bigvee_\lam\bigwedge_{\sigma\geq\lam}\bbA(x_\sigma,y)$ for all $y\in\bbA_0$. We divide the proof into two cases.

\textbf{Case 1}.   $\bbA(x,x)=0$. Then for any $y$, $\bbA(x,y)=0$ and $$ \bigvee_\lam\bigwedge_{\sigma\geq\lam}\bbA(x_\sigma,y) \leq \bigvee_\lam\bigwedge_{\sigma\geq\lam}\bbA(x_\sigma,x_\sigma) =\bbA(x,x) , $$ hence $\bbA(x,y)= \bigvee_\lam\bigwedge_{\sigma\geq\lam}\bbA(x_\sigma,y)=0$.

\textbf{Case 2}. $\bbA(x,x)>0$. Let us be given an $\varepsilon>0$. By Lemma \ref{uniformly continuous}, there is some $\delta>0$ such that $$((\bbA(x,x)+\delta)\ra(\bbA(x,x)-\delta))\&(c-\delta)\geq c-\varepsilon $$  for all $c\leq \bbA(x,x)$.
Since
$$\bigvee_\lam\bigwedge_{\sigma\geq\lam}\bbA(x, x_\sigma)
=\bbA(x,x)=\bw_\lam\bv_{\si\geq\lam}\bbA(x_\si,x_\si),$$
 there exists $\si_0$ such that $$\bbA(x,x)-\delta<\bbA(x,x_\si)\leq\bbA(x_\si,x_\si) <\bbA(x,x)+\delta$$ whenever $\si\geq\si_0$.

For each $\lam$, let $\si$ be an upper bound of $\lam$ and $\si_0$, then
\begin{align*}
\bbA(x,y)&\geq (\bbA(x_\si,x_\si)\ra\bbA(x,x_\si))\&\bbA(x_\si,y)\\
&\geq ((\bbA(x,x)+\delta)\ra(\bbA(x,x)-\delta))\&\bbA(x_\si,y)\\
&\geq  ((\bbA(x,x)+\delta)\ra(\bbA(x,x)-\delta))\&(\bbA(x_\si,y)\wedge\bbA(x,x)-\delta)\\
&\geq \bbA(x_\si,y)\wedge\bbA(x,x)-\varepsilon,
\end{align*} it follows that $$\bbA(x,y)\geq\bbA(x,x)\wedge \bw_{\si\geq\lam}\bbA(x_\si,y)$$ by arbitrariness of $\varepsilon$. Therefore,
\begin{align*}\bbA(x,y)&\geq\bbA(x,x)\wedge \bv_\lam\bw_{\si\geq\lam}\bbA(x_\si,y)\\ &= \Big(\bv_\lam\bw_{\si\geq\lam}\bbA(x_\si,x_\si)\Big)\wedge \Big(\bv_\lam\bw_{\si\geq\lam}\bbA(x_\si,y)\Big)\\ &= \bv_\lam\bw_{\si\geq\lam}\bbA(x_\si,y). & (\bbA(x_\si,y)\leq\bbA(x_\si,x_\si))\end{align*}

The converse inequality is easy since $$\bbA(x,y)=\bbA(x,y)\circ\bbA(x,x)=\bbA(x,y)\circ\bv_\lam\bw_{\si\geq\lam}\bbA(x_\si,x)\leq \bv_\lam\bw_{\si\geq\lam}\bbA(x_\si,y).$$

 This completes the proof.
\end{proof}

In the rest of this section, we show that the Yoneda limit of a forward Cauchy net in a  ${\sf D}(\CQ)$-category $\bbA$ can be equivalently described as the supremum of a certain weight on $\bbA$ in the case $\CQ$ is the interval $[0,1]$ coupled with a continuous t-norm $\&$. This fact is of crucial importance in the study of the relationship between flat completeness and Yoneda completeness.

Let $\CQ$ be a divisible quantale, $\bbA$ be a ${\sf D}(\CQ)$-category, and $\{x_\lam\}$ be a net  in $\bbA$. It is easy to check that the correspondence $$x\mapsto \bigvee_\lam\bigwedge_{\mu\geq\lambda}\mathbb{A}(x,x_\mu)$$ defines a weight $\phi:\mathbb{A}\oto*_{t\phi}$ on $\bbA$ with $t\phi=\bv_\lam\bw_{\mu\geq\lam}\bbA(x_\mu,x_\mu)$, called the weight generated by   $\{x_\lam\}$. Dually, the correspondence $$x\mapsto \bigvee_\lam\bigwedge_{\mu\geq\lambda}\mathbb{A}(x_\mu,x)$$ defines a coweight $\psi:*_{t\psi}\oto\mathbb{A}$ on $\bbA$ with $t\psi=\bv_\lam\bw_{\mu\geq\lam}\bbA(x_\mu,x_\mu)$, called the coweight generated by   $\{x_\lam\}$.

The definition of Yoneda limit can then be rephrased as follows: $a\in\bbA_0$ is a Yoneda limit of  a   net $\{x_\lam\}$ in a ${\sf D}(\CQ)$-category $\bbA$   if   $\bbA(a,a)= \bigvee_\lam\bigwedge_{\sigma\geq\lam}\bbA(x_\sigma, x_\sigma)$  and $\psi(y)= \bbA(a,y)$  for all $y\in \bbA_0$, where   $\psi$ is the coweight on $\bbA$  generated by $\{x_\lam\}$. We shall see that $a$ is a Yoneda limit of  a forward Cauchy net  $\{x_\lam\}$ if and only if $a$ is a supremum of the weight $\phi$ generated by $\{x_\lam\}$.

\begin{lem}\label{psiphileqid} Let $\CQ$ be the interval $[0,1]$ coupled with a continuous t-norm $\&$. If $\{x_\lam\}$ is a forward Cauchy net in a ${\sf D}(\CQ)$-category $\bbA$, then $\psi\circ\phi\leq\bbA$, where $\phi$ and $\psi$ are the weight and the coweight on $\bbA$  generated by $\{x_\lam\}$, respectively. In particular, $$\psi(y)\leq \bbA(-,y)\swarrow\phi = \CP\mathbb{A}(\phi,\y(y))$$ for all $y\in\bbA_0$.  \end{lem}

\begin{proof} Let $a=t\phi=t\psi= \bv_\lam\bw_{\mu\geq\lam}\bbA(x_\mu,x_\mu)$. We must show that $$\psi(y)\circ\phi(x)=(a\ra\psi(y))\&\phi(x)= \psi(y)\&(a\ra\phi(x))
\leq\bbA(x,y)$$ for all $x,y\in\bbA_0$.

Since  $\phi(x),\psi(y)\leq a$,
\begin{align*}
\psi(y)\circ\phi(x)
&=\bv_\lam\bw_{\mu\geq\lambda}(\mathbb{A}(x_\mu,y)\wedge a)\circ\bv_\lam\bw_{\mu\geq\lambda}(\mathbb{A}(x,x_\mu)\wedge a)\\
&\leq\bv_\lam\bw_{\mu\geq\lambda}((\mathbb{A}(x_\mu,y)\wedge a)\circ(\mathbb{A}(x,x_\mu)\wedge a)),
\end{align*} where $(\mathbb{A}(x_\mu,y)\wedge a)\circ(\mathbb{A}(x,x_\mu)\wedge a)=(a\ra(\mathbb{A}(x_\mu,y)\wedge a))\&(\mathbb{A}(x,x_\mu)\wedge a)$. So, it suffices to show that for any $\varepsilon>0$, there is some $\sigma$ such that
$$(\mathbb{A}(x_\mu,y)\wedge a)\circ(\mathbb{A}(x,x_\mu)\wedge a)\leq \mathbb{A}(x,y)+\varepsilon$$ for all $\mu\geq\sigma$. We continue the proof by distinguishing two cases.

\textbf{Case 1}. $a$ is idempotent. In this case, we show that   for any $\mu$, $$((\mathbb{A}(x_\mu,y)\wedge a)\circ(\mathbb{A}(x,x_\mu)\wedge a))\leq\bbA(x,y).$$

If $\bbA(x_\mu,x_\mu)\leq a$, then $\mathbb{A}(x_\mu,y), \mathbb{A}(x,x_\mu)\leq a$, hence
\begin{align*} ((\mathbb{A}(x_\mu,y)\wedge a)\circ(\mathbb{A}(x,x_\mu)\wedge a))&=
(a\ra \mathbb{A}(x_\mu,y))\&\mathbb{A}(x,x_\mu)\\
&\leq (\bbA(x_\mu,x_\mu)\ra \mathbb{A}(x_\mu,y))\&\mathbb{A}(x,x_\mu)\\
&\leq\bbA(x,y).
\end{align*}

If $\bbA(x_\mu,x_\mu)> a$,  we proceed with two subcases.

(i) Either $\mathbb{A}(x_\mu,y)<a$ or $\mathbb{A}(x,x_\mu)<a$. Without loss of generality, suppose  $\mathbb{A}(x_\mu,y)<a$. Then
 \begin{align*}
 ((\mathbb{A}(x_\mu,y)\wedge a)\circ(\mathbb{A}(x,x_\mu)\wedge a))&=
(a\ra \mathbb{A}(x_\mu,y))\&(\mathbb{A}(x,x_\mu)\wedge a)\\
 &\leq(a\ra\mathbb{A}(x_\mu,y))\&\mathbb{A}(x,x_\mu)\\
 &=\mathbb{A}(x_\mu,y)\&\mathbb{A}(x,x_\mu)\\
 &=(\bbA(x_\mu,x_\mu)\ra\mathbb{A}(x_\mu,y))\&\mathbb{A}(x,x_\mu)\\
 &\leq\bbA(x,y),
\end{align*} where, the last equality follows from Equation (\ref{implication around an idempotent}) and  that $\mathbb{A}(x_\mu,y)<a< \bbA(x_\mu,x_\mu)$.

(ii) $\mathbb{A}(x_\mu,y),\mathbb{A}(x,x_\mu)\geq a$. Since $\bbA(x_\mu,x_\mu)\ra\mathbb{A}(x_\mu,y)\geq a$, then   $$(\mathbb{A}(x_\mu,y)\wedge a)\circ(\mathbb{A}(x,x_\mu)\wedge a))= a\leq(\bbA(x_\mu,x_\mu)\ra\mathbb{A}(x_\mu,y))\&\mathbb{A}(x, x_\mu)\leq\bbA(x,y).$$

\textbf{Case 2}.   $a$ is non-idempotent. Let $a_-$ denote  the biggest idempotent element of $\&$  that is smaller than $a$  and $a^+$   the least idempotent element of $\&$  that is bigger than $a$. For any $\varepsilon> 0$ with $[a-\varepsilon,a+\varepsilon]\subset(a_-,a^+)$,
by virtue of Lemma \ref{uniformly continuous}, the   function $$[a-\varepsilon ,a+\varepsilon ]\times[a_-,a^+]\times[0,a^+]\lra [0,1], \quad (d,b,c)\mapsto(d\ra b)\&c$$ is uniformly continuous,  there exists some $\delta\in(0,\varepsilon)$ such that \begin{equation}\label{eq5.3}|((a\ra b)\&c)- ((d\ra b)\&c)|<\varepsilon \end{equation} for all $d\in (a-\delta,a+\delta)$, $b\in[a_-,a^+]$,  and $c\in[0,a^+]$.
 Since $\lim_{\mu}\bbA(x_\mu,x_\mu)=a$,  there is some $\sigma $ such that $\bbA(x_\mu,x_\mu)\in(a-\delta,a+\delta)$ for all $\mu\geq\sigma $.

We claim that $\sigma$ satisfies the requirement, i.e., $$(\mathbb{A}(x_\mu,y)\wedge a)\circ(\mathbb{A}(x,x_\mu)\wedge a)\leq \mathbb{A}(x,y)+\varepsilon$$ for all $\mu\geq\sigma$. The proof is divided into two subcases.

(i)   $\mathbb{A}(x,x_\mu),\mathbb{A}(x_\mu,y)<a_-$.    Then
 \begin{align*}
(\mathbb{A}(x_\mu,y)\wedge a)\circ(\mathbb{A}(x,x_\mu)\wedge a)
&=  (a\ra\mathbb{A}(x_\mu,y))\& \mathbb{A}(x,x_\mu)\\
&=  \mathbb{A}(x_\mu,y)\&\mathbb{A}(x,x_\mu) & (\text{Equation}~ (\ref{implication around an idempotent}))\\
&= (\mathbb{A}(x_\mu,x_\mu)\ra\mathbb{A}(x_\mu,y)) \&\mathbb{A}(x,x_\mu) \\
&\leq  \mathbb{A}(x,y).
\end{align*}

(ii) Either $\mathbb{A}(x,x_\mu)\geq a_-$ or $\mathbb{A}(x_\mu,y)\geq a_-$. Without loss of generality, we suppose that $\mathbb{A}(x_\mu,y)\geq a_-$. Since $\bbA(x,x_\mu)\leq\bbA(x_\mu,x_\mu)\leq a^+$, we have
\begin{align*} ((\mathbb{A}(x_\mu,y)\wedge a)\circ(\mathbb{A}(x,x_\mu)\wedge a))
&= (a\ra \mathbb{A}(x_\mu,y)) \&(\mathbb{A}(x,x_\mu)\wedge a)\\
&\leq  (a\ra\mathbb{A}(x_\mu,y))\&\mathbb{A}(x,x_\mu)\\ &\leq \big((\bbA(x_\mu,x_\mu)\ra\mathbb{A}(x_\mu, y))\&\mathbb{A}(x,x_\mu)\big)
+\varepsilon\\
&\leq  \bbA(x,y)+\varepsilon,
\end{align*} where, the second $\leq$ holds because   of the inequality (\ref{eq5.3}).

The proof is   completed.
 \end{proof}

\begin{thm}\label{sup=Yoneda} Let $\CQ$ be the interval $[0,1]$ coupled with a continuous t-norm $\&$.  If a weight $\phi$  on a ${\sf D}(\CQ)$-category $\bbA$ is  generated by a forward Cauchy net $\{x_\lam\}$ in $\bbA$, then an element $a\in\bbA_0$ is a supremum of $\phi$ if and only if it is a Yoneda limit of $\{x_\lam\}$.\end{thm}

\begin{proof}
\textbf{Sufficiency}. We show that if $a$ is a Yoneda limit of $\{x_\lam\}$ then  $a$ is a supremum of $\phi$. That is, $\bbA(a,y)=\CP\mathbb{A}(\phi,\y(y))$ for all $y\in \bbA_0$.

Let $\psi=\bigvee_\lam\bigwedge_{\mu\geq\lambda}\mathbb{A}(x_\mu,-)$ be the coweight on $\bbA$  generated by $\{x_\lam\}$. By definition of Yoneda limit, it holds that $$\bbA(a,a)= \bv_\lam\bw_{\mu\geq\lambda}\mathbb{A}(x_\mu,x_\mu) =t\phi$$ and that $\bbA(a,y)=\psi(y)$  for all $y\in \bbA_0$. It follows from Lemma \ref{psiphileqid}  that $\bbA(a,y)\leq\CP\mathbb{A}(\phi,\y(y))$ for all $y\in \bbA_0$, so,   it remains to check  that $\CP\mathbb{A}(\phi,\y(y))\leq\bbA(a,y)$ for all $y\in \bbA_0$.

The inequality $\CP\mathbb{A}(\phi,\y(y))\leq\bbA(a,y)$ is trivial if $\bbA(a,y)=\bbA(a,a)$, since $\CP\mathbb{A}(\phi,\y(y))\leq t\phi=\bbA(a,a)$. Now suppose  that $\bbA(a,a)>\bbA(a,y)$. For any $0<\varepsilon<\bbA(a,a)-\bbA(a,y)$, let $\varepsilon_0=\frac{\varepsilon}{2}$. By Lemma \ref{t-norm}, there exists some $0<\delta \leq\frac{\varepsilon}{2}$
such that $$\bbA(a,a)\&\Big((\bbA(a,a)-\delta )\ra(\bbA(a,y)+ \frac{\varepsilon}{2})\Big)<\bbA(a,y) +\frac{\varepsilon}{2}+ \varepsilon_0=\bbA(a,y)+\varepsilon.$$

 By virtue of Equation (\ref{eq1}), we have $$\bbA(a,a)= \bigvee_\lam\bigwedge_{\nu\geq\mu\geq\lam}\bbA(x_\mu,x_\nu).$$ Thus, there is some $\lam_0$ such that $\bbA(x_\mu,x_\nu)>\bbA(a,a)-\delta $ whenever $\nu\geq\mu\geq\lam_0$. In particular, $$\phi(x_{\si}) =\bv_\lam\bw_{\mu\geq\lambda}\mathbb{A}(x_{\si},x_\mu)\geq \bbA(a,a)-\delta $$ for all $\sigma\geq\lam_0$.  Since $$\bbA(a,y)= \bv_\lam\bw_{\mu\geq\lambda}\mathbb{A}(x_\mu,y),$$   there must be  some $\sigma_0\geq\lam_0$ such that $\bbA(x_{\sigma_0},y)<\bbA(a,y)+\delta $.
Therefore, \begin{align*}\CP\mathbb{A}(\phi,\y(y))&= \bbA(y,y)\wedge\bigwedge_{x\in\bbA_0}(\phi(x)\ra\bbA(x,y)) \&t\phi & (\textrm{Equation}~ (\ref{subsethood}))\\ &\leq\bigwedge_{x\in\bbA_0}(\phi(x)\ra\bbA(x,y))\&\bbA(a,a) & (t\phi=\bbA(a,a))\\
&\leq \bbA(a,a)\&(\phi(x_{\si_0})\ra\bbA(x_{\sigma_0},y))\\
&\leq \bbA(a,a)\& ((\bbA(a,a)-\delta )\ra(\bbA(a,y)+\delta ))\\
&\leq \bbA(a,a)\&\Big((\bbA(a,a)-\delta )\ra(\bbA(a,y)+\frac{\varepsilon}{2})\Big)\\
&\leq \bbA(a,y)+\varepsilon.
\end{align*}
Consequently, $\CP\mathbb{A}(\phi,\y(y))\leq\bbA(a,y)$ by arbitrariness of $\varepsilon$.

\textbf{Necessity}. Suppose $a\in\bbA_0$ is a supremum of $\phi$. We show that $a$ is a Yoneda limit of $\{x_\lam\}$. Since $a$ is a supremum of $\phi$, one has that $$\bbA(a,a)=t\phi= \bv_\lam\bw_{\mu\geq\lambda}\mathbb{A}(x_\mu,x_\mu).$$ Thus, it suffices to  show that $\bbA(a,y)= \psi(y)= \bv_\lam\bw_{\mu\geq\lambda}\mathbb{A}(x_\mu,y)$ for all $y\in\bbA_0$.

Since $\CP\mathbb{A}(\phi,\y(y))=\bbA(a,y)$, the inequality $\bbA(a,y)\geq\psi(y)$ follows immediately from Lemma \ref{psiphileqid}. It remains to check that  $\bbA(a,y)\leq\psi(y)$.
The inequality is trivial if $\psi(y)=\bbA(a,a)$.
Suppose that $\psi(y)<\bbA(a,a)$. Given $0<\varepsilon<\bbA(a,a)-\psi(y)$, let $\varepsilon_0=\frac{\varepsilon}{2}$. By Lemma \ref{t-norm}, there is some $0<\delta \leq\frac{\varepsilon}{2}$
such that $$\bbA(a,a)\&\Big((\bbA(a,a)-\delta )\ra(\psi(y)+\frac{\varepsilon}{2})\Big)<\psi(y)+ \frac{\varepsilon}{2}+\varepsilon_0=\psi(y)+\varepsilon.$$
Since $$\bbA(a,a)=t\phi= \bigvee_\lam\bigwedge_{\nu\geq\mu\geq\lam}\bbA(x_\mu, x_\nu),$$ there is some $\lam_0$ such that $\bbA(x_\mu,x_\nu)>\bbA(a,a)-\delta $ whenever $\nu\geq\mu\geq\lam_0$. In particular, $$\phi(x_{\si}) =\bv_\lam\bw_{\mu\geq\lambda}\mathbb{A}(x_{\si},x_\mu)\geq \bbA(a,a)-\delta $$ for all $\sigma\geq\lam_0$. Since $\psi(y)=\bv_\lam\bw_{\mu\geq\lambda}\mathbb{A}(x_\mu,y)$,  there exists some $\si_0\geq\lam_0$ such that $\bbA(x_{\si_0},y)<\psi(y)+\frac{\varepsilon}{2}$. Therefore,
\begin{align*}
\bbA(a,y)&=\CP\mathbb{A}(\phi,\y(y))\\ &\leq \bigwedge_{x\in\bbA_0}(\phi(x)\ra \bbA(x,y))\&\bbA(a,a) & (\textrm{Equation}~ (\ref{subsethood})) \\
&\leq \bbA(a,a)\&(\phi(x_{\si_0})\ra \bbA(x_{\si_0},y))\\
&\leq \bbA(a,a)\&\Big((\mathbb{A}(a,a)-\delta )\ra(\psi(y) +\frac{\varepsilon}{2})\Big)\\
&\leq \psi(y)+\varepsilon,
\end{align*}
so, $\bbA(a,y)\leq \psi(y)$ by arbitrariness of $\varepsilon$.
  \end{proof}

\begin{cor}\label{sup=limit}
Let $\CQ$ be the interval $[0,1]$ coupled with a continuous t-norm $\&$.  If $\{x_\lam\}$ is a biCauchy net  in $\bbA$ and $\phi$ is  the weight generated by $\{x_\lam\}$, then an element $a\in\bbA_0$ is a supremum of $\phi$ if and only if it is a bilimit of $\{x_\lam\}$. \end{cor}
\begin{proof}
Since biCauchy nets are always forward Cauchy,  the conclusion follows immediately from Proposition \ref{LC implies bi-CC} and Theorem \ref{sup=Yoneda}.\end{proof}

\section{Flat completeness implies Yoneda completeness}
A weight $\phi:\bbA\oto *_{t\phi}$  on a ${\sf D}(\CQ)$-category $\bbA$ is  \emph{forward Cauchy} (\emph{biCauchy}, resp.) if it can be generated by a forward Cauchy  (biCauchy, resp.) net in $\bbA$, i.e., there is a forward Cauchy  (biCauchy, resp.) net $\{x_\lam\}$ in $\bbA$ such that $\phi=\bv_\lam\bw_{\mu\geq\lambda}\mathbb{A}(-,x_\mu)$ and $t\phi=\bv_\lam\bw_{\mu\geq\lam}\bbA(x_\mu,x_\mu)$.

\begin{thm}\label{FC weight implies flat weight} If $\CQ$ is the interval $[0,1]$ coupled with a continuous t-norm $\&$, then each flat complete ${\sf D}(\CQ)$-category is Yoneda complete. \end{thm}

\begin{proof} By Theorem \ref{sup=Yoneda}, it suffices to show that each forward Cauchy weight on any  ${\sf D}(\CQ)$-category $\bbA$ is flat.

Suppose that $\phi:\bbA\oto*_{t\phi}$ is generated by a forward Cauchy net $\{x_\lam\}$. Then $t\phi=\lim_{\nu\geq\mu}\bbA(x_\mu,x_\nu)$ and $\phi=\bv_\lam\bw_{\mu\geq\lambda}\mathbb{A}(-,x_\mu)$.

If $t\phi=0$, then $\phi$ is the trivial  weight on $\bbA$, hence flat. It remains to check the conclusion in the case that  $t\phi>0$. That is, for any $b\in[0,1]$, $$\phi\circ-:\CQ\textrm{-}{\bf Dist}(*_b,\bbA)\lra [0,t\phi\wedge b]$$ preserves finite meets.

We'll prove the conclusion in two steps.  First of all, for any $\varepsilon>0$,  by virtue of Lemma \ref{uniformly continuous},   there exists some $\delta>0$ such that   $$((t\phi+\delta)\ra(t\phi-\delta))\&(c-\delta)\geq c-\varepsilon$$ for any $c\leq t\phi$.
Since $\{x_\lam\}$ is forward Cauchy, it follows that there is some $\si_0$ such that $|\bbA(x_\mu,x_\nu)-t\phi|<\delta$ whenever $\si_0\leq\mu\leq\nu$, hence $$\phi(x_{\si}) =\bv_\lam\bw_{\mu\geq\lambda}\mathbb{A}(x_{\si},x_\mu)\geq t\phi-\delta$$ for all $\sigma\geq\si_0$.

\textbf{Step 1}. $\phi\circ-$ preserves the top element. That is,  $\phi\circ(b\wedge t)=t\phi\wedge b$, where $t(x)=\bbA(x,x)$ for all $x\in\bbA_0$.  It suffices to show the  $\phi\circ(b\wedge t)\geq t\phi\wedge b$. In fact,
\begin{align*}
\phi\circ(b\wedge t)&=\bv_x(\bbA(x,x)\ra\phi(x))\&(b\wedge\bbA(x,x))\\
&\geq (\bbA(x_{\si_0},x_{\si_0})\ra\phi(x_{\si_0}))\&(b\wedge\bbA(x_{\si_0},x_{\si_0}))\\
&\geq ((t\phi+\delta)\ra(t\phi-\delta))\&(b\wedge (t\phi-\delta))\\
&\geq ((t\phi+\delta)\ra(t\phi-\delta))\&((t\phi\wedge b)-\delta)\\
&\geq (t\phi\wedge b)-\varepsilon,
\end{align*}so, the inequality $\phi\circ(b\wedge t)\geq t\phi\wedge b$ follows by arbitrariness of $\varepsilon$.

\textbf{Step 2}. $\phi\circ-$ preserves binary meets. That is, $\phi\circ(\psi_1\wedge\psi_2)= (\phi\circ\psi_1)\wedge(\phi\circ\psi_2)$ for any  $\psi_1,\psi_2\in \CQ\textrm{-}{\bf Dist}(*_b,\bbA)$.

Since $\phi=\bv_\lam\bw_{\mu\geq\lambda}\mathbb{A}(-,x_\mu)$, one has  \begin{align*}
(\phi\circ\psi_1)\wedge(\phi\circ\psi_2)
&\leq\bv_\lam\bw_{\mu\geq\lambda}(\mathbb{A}(-,x_\mu)\circ\psi_1)\wedge\bv_\lam\bw_{\mu\geq\lambda}(\mathbb{A}(-,x_\mu)\circ\psi_2)\\
&=\bv_\lam\bw_{\mu\geq\lambda}\psi_1(x_\mu)\wedge\bv_\lam\bw_{\mu\geq\lambda}\psi_2(x_\mu)\\
&=\bv_\lam\bw_{\mu\geq\lambda}(\psi_1(x_\mu)\wedge\psi_2(x_\mu)).
\end{align*}
For each $\lam$, let $\si$ be an upper bound of $\si_0$ and $\lam$. Then $\phi(x_\si)\geq t\phi-\delta$, and
\begin{align*}
\phi\circ(\psi_1\wedge\psi_2)&\geq (\bbA(x_\si,x_\si)\ra\phi(x_{\si}))\&(\psi_1(x_\si)\wedge\psi_2(x_\si))\\
&\geq ((t\phi+\delta)\ra(t\phi-\delta))\&(\psi_1(x_\si)\wedge\psi_2(x_\si)\wedge t\phi)\\
&\geq \psi_1(x_\si)\wedge\psi_2(x_\si)\wedge t\phi-\varepsilon,
\end{align*} hence \[\phi\circ(\psi_1\wedge\psi_2)\geq t\phi\wedge\bw_{\mu\geq\lambda}(\psi_1(x_\mu)\wedge \psi_2(x_\mu))\] by arbitrariness of $\varepsilon$.
Therefore, $$\phi\circ(\psi_1\wedge\psi_2)\geq t\phi\wedge\bv_\lam\bw_{\mu\geq\lambda}(\psi_1(x_\mu)\wedge \psi_2(x_\mu))\geq(\phi\circ\psi_1)\wedge(\phi\circ\psi_2).$$ 

The converse inequality is obvious, hence the conclusion follows.
\end{proof}

\begin{thm}\label{BC implies C} If $\CQ$ is the interval $[0,1]$ coupled with a continuous t-norm $\&$, then Cauchy complete ${\sf D}(\CQ)$-categories are bicomplete.\end{thm}

\begin{proof}  By virtue of Corollary \ref{sup=limit}, it is sufficient to show that each  biCauchy weight on any  ${\sf D}(\CQ)$-category $\bbA$ is Cauchy.

Suppose that $\phi:\bbA\oto*_{t\phi}$ is generated by a biCauchy net $\{x_\lam\}$ in a ${\sf D}(\CQ)$-category $\bbA$. Then $\phi=\bv_\lam\bw_{\mu\geq\lambda}\mathbb{A}(-,x_\mu)$ and $t\phi=\lim_{\mu,\nu}\mathbb{A}(x_\mu,x_\nu)$. Let $a=t\phi$.

If $a=0$, then $\phi$ is the trivial weight on $\bbA$, hence Cauchy. It remains to show that $\phi$ is Cauchy in the case $a>0$. Indeed, we claim that the coweight (generated by $\{x_\lam\}$) $$\psi:*_{a}\oto\bbA, \quad \psi(x)=\bv_\lam\bw_{\mu\geq\lambda}\mathbb{A}(x_\mu,x)$$ is a left adjoint of $\phi$. That is,  $\phi\circ\psi\geq a$ and $\psi\circ\phi\leq \bbA$. The inequality $\psi\circ\phi\leq \bbA$ is proved in Lemma \ref{psiphileqid}. Now  we check that $\phi\circ\psi\geq a$.

For any $\varepsilon>0$, by Lemma \ref{uniformly continuous}, there exists $\delta>0$ such that for any $c\leq a$, $$((a+\delta)\ra(a-\delta))\&(c-\delta)\geq c-\varepsilon.$$
Since $\{x_\lam\}$ is biCauchy, for that $\delta$, there exists some $\si_0$ such that $|\bbA(x_\mu,x_\nu)-a|<\delta$ whenever $\nu,\mu\geq\si_0$,  hence $$\phi(x_{\si}) =\bv_\lam\bw_{\mu\geq\lambda}\mathbb{A}(x_{\si},x_\mu)\geq a-\delta, \quad \psi(x_\si)= \bv_\lam\bw_{\mu\geq\lambda}\mathbb{A}(x_\mu,x_\si)\geq a-\delta$$ for all $\sigma\geq\si_0$. Therefore,
\begin{align*}
\phi\circ\psi&=\bv_x(\phi(x)\circ\psi(x))\\
&\geq(\mathbb{A}(x_{\si_0},x_{\si_0})\ra\phi(x_{\si_0}))\&\psi(x_{\si_0})\\
&\geq ((a+\delta)\ra(a-\delta))\&(a-\delta)\\
&\geq a-\varepsilon,
\end{align*}
hence  $\phi\circ\psi\geq a$ by arbitrariness of $\varepsilon$. \end{proof}


The above theorem shows that each biCauchy weight on a   ${\sf D}(\CQ)$-category is Cauchy. But, the converse conclusion does not hold in general, as shall be seen in Theorem \ref{main result}. 

\section{When does Yoneda completeness imply flat completeness?}
The result in this section shows that Yoneda completeness  does not imply flat completeness in general.   Precisely, it is shown that if $\CQ$ is the unit interval coupled with a continuous t-norm $\&$, then that Yoneda completeness implies flat completeness happens only in the case that $\&$ is either isomorphic to the {\L}ukasiewicz t-norm or to the product t-norm.

\begin{thm}\label{main result} If $\CQ$ is the interval $[0,1]$ coupled with a continuous t-norm $\&$, then the following conditions are equivalent:
\begin{enumerate}[(1)] \item  $\&$ has no non-trivial idempotent elements.
\item  Each non-trivial flat weight on any ${\sf D}(\CQ)$-category   is   forward Cauchy. \item Each non-trivial Cauchy weight on any ${\sf D}(\CQ)$-category   is biCauchy.
\end{enumerate}\end{thm}

\begin{proof}
Suppose that $a$ is a non-trivial idempotent element of $\&$. Take $b\in (a,1)$. Then $$(b\ra a)\&a=\min\{b\ra a,a\}=a.$$ Let $\bbA$ be a ${\sf D}(\CQ)$-category with two objects $x$, $y$, and  $\bbA(x,x)=b$, $\bbA(x,y)=\bbA(y,x)=\bbA(y,y)=0$. The weight $$\phi:\bbA\oto *_a, \quad \phi(x)=a,\ \phi(y)=0$$  is  Cauchy with a left adjoint given by  $$\psi:*_a\oto\bbA, \quad \psi(x)=a,\ \psi(y)=0.$$   But, $\phi$ is not forward Cauchy  since neither of the forward Cauchy weights on $\bbA$,  $\bbA(-,x)$ and $\bbA(-,y)$, is   equal to $\phi$. This proves that $(2)\Rightarrow(1)$ and that $(3)\Rightarrow(1)$.

If $\&$ has no non-trivial idempotent elements, then $\&$ is either isomorphic to the {\L}ukasiewicz t-norm or to the product t-norm. So, to see that  (1) implies both (2) and (3), it is sufficient to show that for both $\CQ=([0,1], \&_P,1)$ and $\CQ=([0,1], \&_{\L},1)$   each non-trivial flat (Cauchy, resp.) weight on any ${\sf D}(\CQ)$-category is forward Cauchy (biCauchy, resp.).   This will be proved separately in the following  \ref{flat1} and \ref{flat module'}.  \end{proof}

\begin{thm}\label{flat1} If $\CQ=([0,1],\&_P,1)$, then each non-trivial flat (Cauchy, resp.) weight on any ${\sf D}(\CQ)$-category is forward Cauchy (biCauchy, resp.).\end{thm}

\begin{proof}
Since the quantale $([0,1],\&_P,1)$   is isomorphic to  Lawvere's quantale $([0,\infty]^{\rm op},+,0)$,  and ${\sf D}(\CQ)$-categories for $\CQ=([0,\infty]^{\rm op},+,0)$ are exactly generalized partial metric spaces (Example \ref{PMS}),  it suffices to show that  each non-trivial flat (Cauchy, resp.)  weight on any generalized partial metric space is forward Cauchy (biCauchy, resp.). This will be done in the following   \ref{flat module} and \ref{cauchy=bicauchy}. \end{proof}

To avoid unnecessary notations, we simply write $X$ for a generalized partial metric space $(X,p)$ and write $X(x,y)$ for $p(x,y)$. A weight $\phi:X\oto*_{t\phi}$ (of type $t\phi$) on a generalized partial metric space $X$ is, by definition, a map $\phi: X\lra [0,\infty]$ such that
$$X(x,y)+(\phi(y)-X(y,y))\geq \phi(x) \quad \textrm{and} \quad \phi(x)\geq \max \{X(x,x),t\phi\}$$ for all $x,y\in X$.
Dually, a coweight $\psi:*_{t\psi} \oto X$ on  $X$ is a map $\psi: X\lra [0,\infty]$ such that $$X(x,y)+(\psi(x)-X(x,x))\geq \psi(y)\quad \textrm{and} \quad \psi(x)\geq\max\{X(x,x),t\psi\}$$ for all $x,y\in X$.

For a weight $\phi$ and a coweight $\psi$ on a generalized partial metric space $X$, we have $$\phi\circ \psi=\inf_{x}(\phi(x)+(\psi(x)-X(x,x))), $$ agreeing again that $\infty-\infty=0$.

\begin{lem}\label{GFC is FC}
Let $\{x_\lam\}$ be a net in a generalized partial metric space $X$. If $$\lim_{\nu\geq\mu}(X(x_\mu,x_\nu)-X(x_\mu,x_\mu))=0,$$ then $\{x_\lam\}$ is forward Cauchy. \end{lem}

\begin{proof} By Remark \ref{FC=convergent}(1), it suffices to show that the limit $\lim_{\nu\geq\mu}X(x_\mu,x_\nu)$ exists ($\infty$ is allowed). If there exists  some $\lam$ such that $X(x_\mu,x_\mu)=\infty$  for all $\mu\geq\lam$, then it is trivial that $\lim_{\nu\geq\mu}X(x_\mu,x_\nu)=\infty$.
Now, suppose that  for all $\lam$, there exists some $\mu\geq\lam$ such that $X(x_\mu,x_\mu)<\infty$. We prove, in this case, that the limit $\lim_{\nu\geq\mu}X(x_\mu,x_\nu)$ exists in three steps.

\textbf{Step 1}.  The net $\{X(x_\mu,x_\mu)\}$ is eventually bounded. Since $$\lim_{\nu\geq\mu}(X(x_\mu,x_\nu)-X(x_\mu,x_\mu))=0,$$ there exists some $\lambda$ such that $X(x_\lambda,x_\lambda)<\infty$ and that $X(x_\mu,x_\nu)-X(x_\mu,x_\mu)<1$ whenever $\nu\geq\mu\geq\lambda$. If $\{X(x_\mu,x_\mu)\}$ is not eventually bounded, there would exist $\mu\geq\lambda$ such that $X(x_\mu,x_\mu)\geq X(x_\lambda,x_\lambda)+2$. So, $$ X(x_{\lambda},x_\mu)-X(x_{\lambda},x_{\lambda})\geq X(x_\mu,x_\mu)-X(x_{\lambda},x_{\lambda})\geq 2,$$ a contradiction.

\textbf{Step 2}. The limit $\lim_\lambda X(x_\lambda,x_\lambda)$ exists. Otherwise, $\{X(x_\lambda,x_\lambda)\}$  would have two cluster points, say, $A$ and $B$. Suppose that $A<B$ and let $\varepsilon=\frac{A-B}{3}$. It follows from   $$\lim_{\nu\geq\mu}(X(x_\mu,x_\nu)-X(x_\mu,x_\mu))=0 $$ that there exists some $\lambda_0$ such that $X(x_\mu,x_\nu)-X(x_\mu,x_\mu)<\varepsilon$ whenever $\nu\geq\mu\geq\lambda_0$.  Because $A,B$ are cluster points of $\{X(x_\lambda,x_\lambda)\}$, there exist some $\nu_0, \mu_0$ such that
$\nu_0\geq\mu_0\geq\lambda_0$, $|X(x_{\mu_0},x_{\mu_0})-A|<\varepsilon$ and   $|X(x_{\nu_0},x_{\nu_0})-B|<\varepsilon.$  Then  $X(x_{\nu_0},x_{\nu_0})-X(x_{\mu_0},x_{\mu_0})\geq \varepsilon$, a contradiction to that
$$X(x_{\nu_0},x_{\nu_0})-X(x_{\mu_0},x_{\mu_0})\leq X(x_{\mu_0},x_{\nu_0})-X(x_{\mu_0}, x_{\mu_0})<\varepsilon.$$

\textbf{Step 3}. The limit $\lim_{\nu\geq\mu}X(x_\mu,x_\nu)$ exists. Let $A=\lim_\lambda X(x_\lambda,x_\lambda)$. Then
$$|X(x_\mu,x_\nu)-A|\leq |X(x_\mu,x_\nu)-X(x_\mu,x_\mu)| + |X(x_\mu,x_\mu)-A|$$ for all $\nu\geq\mu$, showing that
$\lim_{\nu\geq\mu}X(x_\mu,x_\nu)=A$.   \end{proof}

\begin{prop}\label{flat module} Let $\phi$ be a non-trivial weight on a generalized partial metric space $X$. The  following   are equivalent: \begin{enumerate}[\rm(1)]
\item
$\phi$  is flat.
\item  $\phi$ satisfies the following conditions: \begin{enumerate}
\item[\rm(a)]$\inf_{\phi(x)<\infty}(\phi(x)-X(x,x))=0$;
\item[\rm(b)]if $\phi(x_i)<X(x_i,x_i)+\delta_i\ (i=1,2)$, then there exist some $y\in X$ and $\varepsilon>0$ such that $\phi(y)<X(y,y)+\varepsilon$ and $X(x_i,y)+\varepsilon<X(x_i,x_i)+\delta_i\ (i=1,2)$.
\end{enumerate}
\item  $\phi$ is forward Cauchy.
\end{enumerate}\end{prop}

\begin{proof} $(1)\Rightarrow(2)$ By virtue of \textbf{Fact 2} at the end of Section \ref{section 4},
for each $b\geq0$, $$\inf_{x}(\phi(x)+ (\max\{b,X(x,x)\}-X(x,x))) =\max\{b,t\phi\}.$$ Since $\phi$ is non-trivial, one has   $t\phi<\infty$. Applying the above equality to   $b=t\phi$  gives that $$\inf_{\phi(x)<\infty}\big(\phi(x)-X(x,x)+ \max\{t\phi,X(x,x)\}\big) =t\phi.$$ Hence $$\inf_{\phi(x)<\infty}(\phi(x)-X(x,x))=0.$$ This proves (a).

Now we prove  (b). Let $\lambda_1=\delta_2-X(x_1,x_1)$ and $\lambda_2=\delta_1-X(x_2,x_2)$, and let $k$ be a positive real number such that $\lambda_i+k>0$.
Then $$\max\{\lambda_1+k+\phi(x_1), \lambda_2+k+\phi(x_2)\}<\delta_1+ \delta_2+k.$$  Consider the coweights $\psi_1,\psi_2:*_0\oto X$  on $X$  (of type $0$) given by  $$\psi_1(y)=\lambda_1+k+X(x_1,y), \quad   \psi_2(y)=\lambda_2+k+X(x_2,y) $$  for all $y\in X$. Let $\psi_1\wedge \psi_2$ be the  meet of $\psi_1$ and $\psi_2$ in the complete lattice of  all distributors $*_0\oto X$. Since $\phi$ is flat, it follows that
\begin{align*}
&~ \inf_{x}(\max\{\psi_1(x),\psi_2(x)\}+(\phi(x)-X(x,x)))\\ =&~ \phi\circ(\psi_1\wedge\psi_2)\\ = & ~ \max\{\phi\circ\psi_1,\phi\circ\psi_2\}  \\
= & ~\max\Big\{\inf_{x}(\psi_1(x)+(\phi(x)-X(x,x))),  \inf_{x}(\psi_2(x)+(\phi(x)-X(x,x)))\Big\}\\
= & ~\max\{\lambda_1+k+\phi(x_1),\lambda_2+k+\phi(x_2)\}\\
< &~\delta_1+\delta_2+k.\end{align*}
Thus, there is some $y\in X$ such that $$\lambda_i+k+X(x_i,y)+\phi(y)-X(y,y)<\delta_1+\delta_2+k,$$ hence $$X(x_i,y)-X(x_i,x_i)+\phi(y)-X(y,y)<\delta_i.$$
Find some $\alpha_i,\beta_i>0 \ (i=1,2)$ such that $X(x_i,y)-X(x_i,x_i)<\alpha_i$, $\phi(y)-X(y,y)<\beta_i$, and that $\alpha_i+\beta_i\leq\delta_i$. Let $\varepsilon=\min\{\beta_1,\beta_2\}$. Then  $$\phi(y)<X(y,y)+\varepsilon\quad \textrm{and } \quad X(x_i,y)+\varepsilon<X(x_i,x_i)+\delta_i$$ as desired.

$(2)\Rightarrow(3)$ First of all, since $$\inf_{\phi(x)<\infty}(\phi(x)-X(x,x))=0,$$
for each $\varepsilon>0$ there exists some $x$ such that $\phi(x)<X(x,x)+\varepsilon$.

Consider the set $$D=\{(x_\lambda,r_\lambda)\mid \phi(x_\lam)<X(x_\lambda,x_\lambda)+r_\lambda\}.$$ Define a binary relation $\sqsubseteq$ on $D$ by $$(x_\mu,r_\mu)\sqsubseteq (x_\nu, r_\nu)\iff X(x_\mu,x_\nu)+r_\nu\leq X(x_\mu,x_\mu)+r_\mu$$ for all $(x_\mu,r_\mu), (x_\nu, r_\nu)\in D$. In particular, if $(x_\mu,r_\mu)\sqsubseteq (x_\nu, r_\nu)$ then $r_\nu\leq r_\mu$. It is easily seen that $(D,\sqsubseteq)$ is a directed set.
We shall show that the net $$x:D\lra X, \quad (x_\lambda,r_\lambda)\mapsto x_\lambda $$ is forward Cauchy and generates $\phi$.

\textbf{Step 1}.  $\{x_{(x_\lambda,r_\lambda)}\}$ is  forward Cauchy. Given $\varepsilon>0$, take some $(x_\lambda,r_\lambda)\in D$ with  $r_{\lambda}<\varepsilon$. Then  $$X(x_\mu,x_\nu)-X(x_\mu,x_\mu)\leq r_\mu-r_\nu<\varepsilon-r_\nu<\varepsilon$$ whenever $(x_\nu,r_\nu)\sqsupseteq (x_\mu, r_\mu)\sqsupseteq(x_\lambda,r_\lambda)$. Thus,  $\{x_{(x_\lambda,r_\lambda)}\}$  is forward Cauchy by Lemma \ref{GFC is FC}.

\textbf{Step 2}. $\phi$ is generated by $\{x_{(x_\lambda,r_\lambda)}\}$, that is, $$\phi(x)=\inf_{(x_\lambda,r_\lambda)} \sup_{(x_\sigma, r_\sigma)\sqsupseteq(x_\lambda,r_\lambda)} X(x, x_\sigma)$$ for all $x\in X$. We proceed with two cases.

\textbf{Case 1}. $\phi(x)=\infty$.   For each $(x_\sigma, r_\sigma)\in D$, since $$X(x,x_\sigma)+(\phi(x_\sigma)-X(x_\sigma,x_\sigma))\geq \phi(x),$$ it follows that $X(x,x_\sigma)=\infty$, hence the equality is trivial.

\textbf{Case 2}. $\phi(x)<\infty$. Take   $\varepsilon>0$.  Since $(x,\phi(x)-X(x,x)+\varepsilon)\in D$,  there exists some index $\nu$ such that $(x_\nu,r_\nu)=(x,\phi(x)-X(x,x)+\varepsilon)$. Since for each $(x_\sigma, r_\sigma)\sqsupseteq(x_\nu,r_\nu)$,   $$X(x,x_\si)\leq X(x,x)+ \phi(x)-X(x,x)+\varepsilon-r_\si=\phi(x)-r_\si+ \varepsilon<\phi(x)+\varepsilon,$$
it follows that $$\sup_{(x_\sigma, r_\sigma)\sqsupseteq(x_\nu,r_\nu)}X(x,x_\sigma)\leq\phi(x)+ \varepsilon,$$
then $$\inf_{(x_\lambda,r_\lambda)}\sup_{(x_\sigma, r_\sigma)\sqsupseteq(x_\lambda, r_\lambda)}X(x, x_\sigma)\leq \phi(x) $$ by arbitrariness of $\varepsilon$.

Conversely, let $(x_\lambda,r_\lambda)\in D$ and  $\varepsilon>0$. By (a), there exists some $y\in X$ such that $\phi(y)<X(y,y)+\varepsilon$. Thus, there is some index $\nu$ such that  $(y,\varepsilon)=(x_\nu,r_\nu) \in D$. Let $(x_\si,r_\si)$ be an upper bound of $(x_\nu,r_\nu)$ and $(x_\lambda,r_\lambda)$ in $D$.  Then
\begin{align*} \phi(x)&\leq \phi(x_\sigma)+X(x,x_\sigma)- X(x_\sigma,x_\sigma) &(\textrm{$\phi$ is a weight})\\ &\leq X(x,x_\sigma)+ r_\sigma & (\textrm{$(x_\si,r_\si)\in D$})\\ &\leq X(x,x_\sigma)+\varepsilon, & (\textrm{$r_\sigma\leq r_\nu= \varepsilon$})\end{align*}
it follows that $$\phi(x)\leq\sup_{(x_\sigma, r_\sigma)\sqsupseteq(x_\lambda,r_\lambda)}X(x,x_\sigma) $$ by arbitrariness of $\varepsilon$. Therefore, $$\phi(x)\leq \inf_{(x_\lambda,r_\lambda)}\sup_{(x_\sigma, r_\sigma)\sqsupseteq(x_\lambda,r_\lambda)}X(x, x_\sigma).$$

$(3)\Rightarrow(1)$   Theorem \ref{FC weight implies flat weight}. \end{proof}

\begin{prop}\label{cauchy=bicauchy} Each non-trivial Cauchy weight on a generalized partial metric space is biCauchy. \end{prop}
\begin{proof}This follows from the fact that each non-trivial Cauchy weight on a generalized partial metric space is generated by a biCauchy sequence, as shown in \cite{Pu2012}, Proposition 4.10. \end{proof}

\begin{thm}\label{flat module'} If $\CQ$ is the interval $[0,1]$ coupled with  the {\L}ukasiewicz t-norm, then each non-trivial flat (Cauchy, resp.)  weight on any ${\sf D}(\CQ)$-category is forward Cauchy (biCauchy, resp.). \end{thm}

Theorem \ref{flat module'}   can be proved in a way similar to that of propositions \ref{flat module} and \ref{cauchy=bicauchy}. For instance, it can be proved that for a non-trivial weight $\phi$ on a ${\sf D}(\CQ)$-category $\bbA$, the following conditions are equivalent:
\begin{enumerate}[\rm(1)]
\item
$\phi$ is flat.
\item  $\phi$ satisfies the following conditions: \begin{enumerate}
\item[\rm(a)]$\bigvee_{\phi(x)>0}(1+\phi(x)-\bbA(x,x))=1$;
\item[\rm(b)]if $\bbA(x_i,x_i)\&\delta_i<\phi(x_i)\ (i=1,2)$, then there exist some $y\in \bbA_0$ and $\varepsilon<1$ such that $\bbA(y,y)\&\varepsilon<\phi(y)$ and $\bbA(x_i,x_i)\&\delta_i<\bbA(x_i,y)\&\varepsilon\ (i=1,2)$.
\end{enumerate}
\item $\phi$ is forward Cauchy.\end{enumerate}

In the following we sketch a  proof of Theorem \ref{flat module'} with help of Theorem \ref{flat1}.  A generalized partial metric space $X$ is  \emph{ bounded} if $X(x,y)\leq 1$ for all $x, y\in X$. A (co)weight  $\phi$ ($\psi$, resp.)   on $X$ is \emph{ bounded}  if $\phi(x)\leq1$ ($\psi(x)\leq1$, resp.) for all $x\in X$.

Let $\CQ$ be the interval $[0,1]$ coupled with  the {\L}ukasiewicz t-norm. A ${\sf D}(\CQ)$-category $\bbA$ is a set $\mathbb{A}_0$ together with a map $\mathbb{A}:\mathbb{A}_0\times\mathbb{A}_0\lra [0,1]$ such that  $\bbA(x,y)\leq \bbA(x,x)\wedge\bbA(y,y)$ and that $\bbA(y,z)+\bbA(x,y)-\bbA(y,y)\leq\bbA(x,z)$ for all $x, y, z\in\bbA_0$. For a ${\sf D}(\CQ)$-category $\bbA$, if we let $X(x,y)=1-\bbA(x,y)$ for all $x,y\in\bbA_0$, then   $X$ becomes a bounded generalized partial metric space (with underlying set $\bbA_0$), called the \emph{associated generalized partial metric space} of $\bbA$. Furthermore,   if $\phi$ is a (co)weight on $\bbA$ with type $t\phi$, then $\phi'=1-\phi$  is a bounded (co)weight on $X$ with type  $t\phi'=1-t\phi$.

We leave the lengthy but not difficult verification of the following proposition to the reader.

 \begin{prop}\label{phi and phi'} Let $\CQ$ be the interval $[0,1]$ coupled with  the {\L}ukasiewicz t-norm, $\bbA$  a ${\sf D}(\CQ)$-category, and   $X$   the associated generalized partial metric space of $\bbA$. Then for each non-trivial weight $\phi$ on $\bbA$, we have:
\begin{enumerate}[\rm(1)]
  \item $\phi$ is a flat weight on $\bbA$ if and only if $\phi'$ is a flat weight on $X$.
  \item  $\phi$ is a forward Cauchy weight on $\bbA$ if and only if $\phi'$ is a forward Cauchy weight on $X$.
  \item $\phi$ is a Cauchy weight on $\bbA$ if and only if $\phi'$ is a Cauchy weight on $X$.
  \item  $\phi$ is a biCauchy weight on $\bbA$ if and only if $\phi'$ is a biCauchy weight on $X$.
\end{enumerate}\end{prop}

\begin{proof}[Proof of Theorem \ref{flat module'}] The conclusion follows from a combination of  propositions \ref{flat module},  \ref{cauchy=bicauchy}, and \ref{phi and phi'}. \end{proof}

Now we come to the main result in this section.

\begin{thm}\label{main result II} Let $\CQ$ be the interval $[0,1]$ coupled with a continuous t-norm $\&$. Then the following conditions are equivalent:
\begin{enumerate}[(1)] \item  $\&$ is either isomorphic to the {\L}ukasiewicz t-norm or to the product t-norm. 
\item  Each Yoneda complete ${\sf D}(\CQ)$-category with an isolated element is flat complete.
\item  Each Yoneda complete ${\sf D}(\CQ)$-category with an isolated element is Cauchy complete.
\item  Each bicomplete ${\sf D}(\CQ)$-category with an isolated element is Cauchy complete.
\end{enumerate}\end{thm}
\begin{proof} That $(1)\Rightarrow(4)$ and that $(1)\Rightarrow(2)\Rightarrow(3)$  follows from Theorem \ref{main result}, Theorem \ref{sup=Yoneda}, Corollary \ref{sup=limit} and Proposition \ref{Cauchy is flat}. It remains to check that $(3)\Rightarrow(1)$ and that $(4)\Rightarrow(1)$. Suppose on the contrary that $\&$ is neither isomorphic to the   {\L}ukasiewicz t-norm nor to the product t-norm. Then, by virtue of Theorem \ref{ord sum}, $\&$ has a non-trivial idempotent element, say $a$. Consider the ${\sf D}(\CQ)$-category $\bbA$ and the Cauchy weight $\phi$ given in the proof of Theorem \ref{main result}. It is clear that $\phi$ does not have a supremum, hence $\bbA$ is not Cauchy complete. But, $\bbA$ is  Yoneda complete (hence bicomplete) and has an isolated element  $y$, a contradiction. \end{proof}

Suppose $\CQ$ is the interval $[0,1]$ coupled with a continuous t-norm $\&$.  Theorem \ref{main result} shows that, for ${\sf D}(\CQ)$-categories, Cauchy weights  need not be biCauchy, though biCauchy weights are always Cauchy. However, Theorem 4.19 in Hofmann and Reis \cite{Hofmann2013} shows that a weight on a $\CQ$-category (not a ${\sf D}(\CQ)$-category) is Cauchy if and only if it is biCauchy.  The following conclusion is a bit of a surprise compared with the result of Hofmann and Reis.

\begin{prop}If $\CQ$ is the interval $[0,1]$ coupled with a continuous t-norm $\&$, then the following conditions are equivalent:
\begin{enumerate}[(1)] \item  $\&$ has no non-trivial idempotent elements.
\item  Each  flat weight on any $\CQ$-category   is   forward Cauchy.
\end{enumerate} \end{prop}
\begin{proof} $(1)\Rightarrow(2)$ is contained in  Theorem \ref{main result} because, in this case, a $\CQ$-category is exactly a global  ${\sf D}(\CQ)$-category. When $\&$ is the product t-norm, an equivalent version of this implication, i.e., in the form of generalized metric spaces, is contained in Vickers \cite{Vickers2005}.  It remains to check $(2)\Rightarrow(1)$.

Suppose on the contrary that $\&$ has a non-trivial idempotent element, say, $a$. Consider the $\CQ$-category $\bbA=([0,1],\rightarrow)$ and the weight $\phi$ on $\bbA$ given by $\phi(0)=1$ and $\phi(x)=a$ for all $x\in(0,1]$. It is easy to check that $\phi$ is flat. But, $\phi$ is not forward Cauchy. To see this, suppose that $$\phi=\bv_\lam\bw_{\si\geq\lam}\bbA(-,x_\si)$$ for a forward Cauchy net $\{x_\lam\}$. Then $$a=\phi(1)=\bv_\lam\bw_{\si\geq\lam}\bbA(1,x_\si)=\bv_\lam\bw_{\si\geq\lam}x_\si.$$  Thus, there exist $0<\varepsilon<a$ and $\lam_0$ such that $x_\si>a-\varepsilon$ whenever $\si\geq\lam_0$. Therefore, for each $x\in(0,a-\varepsilon)$, one has $$\phi(x)=\bv_\lam\bw_{\si\geq\lam}\bbA(x,x_\si)=1,$$ a contradiction to that $\phi(x)=a$.
 \end{proof}

\section{Concluding remarks}
Let $\CQ$ be the interval $[0,1]$ coupled with a continuous t-norm $\&$. In this paper we considered four kinds of completeness for ordered fuzzy sets valued in $\CQ$: Yoneda completeness, bicompleteness, flat completeness, and Cauchy completeness. Both Yoneda completeness and flat completeness are natural extensions of directed completeness of partially ordered sets; bicompleteness and Cauchy completeness are weak (or, ``symmetric") version  of Yoneda completeness and flat completeness, respectively.

This paper is concerned with the relationship among these kinds of completeness. The main results are:
\begin{enumerate}[(1)]
\item Flat completeness is stronger than the other three; bicompleteness  is weaker than the other three.
$$\bfig
\square<1100, 500>[{\rm flat~complete}`{\rm Yoneda~ complete}`{\rm Cauchy~complete}`{\rm bicomplete};` ` ` ]
\efig$$
\item All Yoneda  complete (bicomplete, resp.) ordered fuzzy sets with an isolated element are flat (Cauchy, resp.) complete  if and only if the continuous t-norm $\&$ is either isomorphic to the {\L}ukasiewicz t-norm or to the product t-norm. \end{enumerate}

The results provide an example of an interesting phenomenon in fuzzy set theory: the  properties of mathematical structures   (here, ordered fuzzy sets) interact with the properties of the table  of truth-values. This interaction has no counterpart in classic mathematics, so, the study of this interaction is very likely to be an important topic in  fuzzy set theory. More examples in this regard can be found in \cite{Chen2011,Lai2005} (fuzzy topology) and \cite{Lai2016} (fuzzy order).

Though  the results presented here are proved only in the case that the quantale  $\CQ$ is the interval $[0,1]$ coupled with a continuous t-norm $\&$, the notions introduced here make sense for any quantale. So, a  question arises:

\begin{ques} For what kind of quantales $\CQ=(Q,\&,1)$ do we still have Theorem \ref{sup=Yoneda} and Theorem \ref{FC weight implies flat weight}? In particular, are they true for all BL-algebras in the sense of H\'{a}jek \cite{Hajek1998}?
\end{ques}

We would like to note that if $\CQ=(Q,\&, 1)$ is a frame, i.e. $\&=\wedge$, then it is not hard to check that all forward Cauchy (biCauchy, resp.) weights on any ${\sf D}(\CQ)$-category are flat (Cauchy, resp.), but, we do not know whether Theorem \ref{sup=Yoneda} still holds in this case.

~

\noindent\textbf{Acknowledgements}

The  authors thank gratefully the reviewers for their insights and comments that helped to improve  this paper greatly.

\end{document}